\documentclass[12pt,a4paper]{article}
\usepackage[T1]{fontenc}
\usepackage{lmodern}
\usepackage[utf8]{inputenc}
\usepackage{amsthm}
\usepackage{amsmath,amssymb}
\usepackage{graphics}
\usepackage{todonotes}
\usepackage{lscape,graphicx}
\usepackage{enumitem}
\usepackage{amsfonts}
\usepackage[T1]{fontenc}
\usepackage{longtable}
\usepackage{authblk}
\usepackage{lipsum}
\usepackage{epsfig}
\usepackage{scalerel}
\usepackage{float}
\usepackage{hyperref}
\usepackage{comment}
\usepackage[normalem]{ulem}
\usepackage{multicol}
\usepackage{fancyhdr}
\usepackage[english]{babel}
\usepackage{mathrsfs}
\usepackage{tocloft}
\usepackage{xcolor}
\usepackage{titletoc}
\usepackage{mathtools}
\usepackage[toc,page]{appendix}

\addto{\captionsenglish}{}
\makeatletter
\def\thm@space@setup{%
  \thm@preskip=1cm plus .5cm minus .5cm
  \thm@postskip=.5cm plus .6cm minus .5cm 
}

\makeatother

\newtheorem{thm}{Theorem}

\newtheorem{lma}{Lemma}

\newtheorem{cor}{Corollary}

\newtheorem{rmk}{Remark}

\numberwithin{thm}{section}
\numberwithin{lma}{section}
\numberwithin{dfn}{section}
\numberwithin{cor}{section}
\numberwithin{rmk}{section}
\numberwithin{prop}{section}

\def\mfq{{\mathfrak q}}
\def\mfd{{\mathfrak d}}
\def\mfm{{\mathfrak m}}
\def\mfc{{\mathfrak C}}
\def\mfa{{\mathfrak A}}
\def\mcm{{\mathcal M}}
\def\mcp{{\mathcal P}}
\def\mfp{{\mathfrak p}}

\newcommand*{\thmref}[1]{Theorem~\ref{#1}}

\newcommand*{\lmaref}[1]{Lemma~\ref{#1}}

\title{On the distribution of the number of distinct generators of $h$-free and $h$-full elements in an abelian monoid}

\date{}

\begin{document}





  



\author[1]{Sourabhashis Das}
\affil[1]{Corresponding author, Department of Pure Mathematics, University of Waterloo, 200 University Avenue West, Waterloo, Ontario, Canada, N2L 3G1, \textit{s57das@uwaterloo.ca}.}
%
%
\author[2]{Wentang Kuo}
\affil[2]{Department of Pure Mathematics, University of Waterloo, 200 University Avenue West, Waterloo, Ontario, Canada, N2L 3G1, \textit{wtkuo@uwaterloo.ca}.}
\author[3]{Yu-Ru Liu}
\affil[3]{Department of Pure Mathematics, University of Waterloo, 200 University Avenue West, Waterloo, Ontario, Canada, N2L 3G1, \textit{yrliu@uwaterloo.ca}.}

%
%
%

\maketitle 

\begin{abstract}
This work introduces the first in-depth study of $h$-free and $h$-full elements in abelian monoids, providing a unified approach for understanding their role in various mathematical structures. Let $\mfm$ be an element of an abelian monoid, with $\omega(\mfm)$ denoting the number of distinct prime elements generating $\mfm$. We study the moments of $\omega(\mfm)$ over subsets of $h$-free and $h$-full elements, establishing the normal order of $\omega(\mathfrak{m})$ within these subsets. Our findings are then applied to number fields, global function fields, and geometrically irreducible projective varieties, demonstrating the broad relevance of this approach.
\end{abstract}

\section{Introduction}

For a natural number $n$, let the prime factorization of $n$ be given as
\begin{equation*}
n = p_1^{s_1} \cdots p_r^{s_r},
\end{equation*}
where $p_i$'s are its distinct prime factors and $s_i$'s are their respective multiplicities. Let $\omega(n)$ denote the total number of distinct prime factors in the factorization of $n$. Thus, $\omega(n) =r$. Let $h \geq 2$ be an integer. We say $n$ is $h$\textit{-free} if $s_i \leq h-1$ for all $i \in \{1, \cdots, r\}$, and we say $n$ is $h$\textit{-full} if $s_i \geq h$ for all $i \in \{1, \cdots, r\}$. The set of $h$-free and the set of $h$-full numbers are well-studied in the literature (see \cite{dkl2}, \cite{dkl4}, \cite{is}, \cite{jala}, and \cite{jala2}). The former has a density of $1/\zeta(h)$, where $\zeta(s)$ is the classical Riemann $\zeta$-function, and the latter is a zero-density subset of the natural numbers. Moreover, every natural number that is not an $h$-free or an $h$-full number can be written uniquely as a product of two co-prime numbers, one being $h$-free, and the other being $h$-full. Thus, to have a more refined understanding of the distributions of $\omega(n)$, we study its corresponding distributions over the $h$-free numbers and over the $h$-full numbers. 
For $x >1$, let $S_h(x)$ and $N_h(x)$ respectively be the set of $h$-free and the set of $h$-full numbers less than or equal to $x$. The authors in \cite[Theorem 1.1]{dkl2} proved the first and the second moments of $\omega(n)$ over $h$-free elements as the following:
$$\sum_{n \in S_h(x)} \omega(n) = \frac{1}{\zeta(h)} x \log \log x + 
 \frac{C_1}{\zeta(h)} x
+ O_h \left( \frac{x}{\log x}\right),$$
and
\begin{align*}
\sum_{n \in S_h(x)}  \omega^2(n) 
& = \frac{1}{\zeta(h)}  x (\log \log x)^2 + \frac{2 C_1 + 1}{\zeta(h)} x \log \log x + \frac{C_2}{\zeta(h)} x + O_h \left( \frac{x \log \log x}{\log x}\right),
\end{align*}
where $C_1$ and $C_2$ are constants that depend only on the set of primes and $h$, and where $O_X$ denotes that the implied big-O constant depends on the variable set $X$. Similarly, in \cite[Theorem 1.2]{dkl2}, they proved the first and the second moments of $\omega(n)$ over $h$-full elements as the following:
\begin{align*}
\sum_{n \in N_h(x)} \omega(n) & =  \gamma_{0,h} x^{1/h} \log \log x + D_1 \gamma_{0,h} x^{1/h} + O_h \left( \frac{x^{1/h}}{\log x} \right),
\end{align*}
and
\begin{align*}
     \sum_{n \in N_h(x)} \omega^2(n)
     & = \gamma_{0,h} x^{1/h} (\log \log x)^2 +   \left( 2 D_1 + 1 \right) \gamma_{0,h} x^{1/h} \log \log x + D_2 \gamma_{0,h} x^{1/h} \\
     & \hspace{.5cm} + O_h \left( \frac{x^{1/h} \log \log x}{\log x} \right),
\end{align*}
where $D_1, D_2$, and $\gamma_{0,h}$ are constants that depend only on the set of primes and $h$. 

In this article, we expand the study of $h$-free and $h$-full elements to include any countably generated free abelian monoids for the first time. This new approach allows us to analyze a variety of mathematical objects, such as number fields, global function fields, and geometrically irreducible projective varieties, all within one unified framework. Section \ref{applications} explores various applications of this generalized setting, demonstrating the broad applicability of our approach.

Let $\mathcal{P}$ be a countable set of elements with a map
$$N: \mathcal{P} \rightarrow \mathbb{Z}_{>1}, \quad \mfp \mapsto N(\mfp).$$
We call the elements of $\mathcal{P}$ \textit{prime elements}, and the map $N(\cdot)$ as the \textit{Norm map}. Let $\mathcal{M}$ be a \textit{free abelian monoid} generated by elements of $\mathcal{P}$. In other words, for each $\mfm \in \mathcal{M}$, we write
$$\mfm = \sum_{\mfp \in  \mathcal{P}} n_\mfp(\mfm) \mfp,$$
with $n_\mfp(\mfm) \in \mathbb{N} \cup \{ 0 \}$ and $n_\mfp(\mfm) =0$ for all but finitely many $\mfp$. We call $n_\mfp(\mfm)$ the \textit{multiplicity} of $\mfp$ in $\mfm$. We extend the norm map $N$ on $\mcm$ as the following:
\begin{align*}
    N : \mcm & \rightarrow \mathbb{N} \\
    \mfm = \sum_{\mfp \in \mathcal{P}} n_\mfp(\mfm) \mfp & \longmapsto N(\mfm) := \prod_{\mfp \in \mathcal{P}} N(\mfp)^{n_\mfp(\mfm)}. 
\end{align*}
Thus, $N$ can be extended to a monoid homomorphism from $(\mcm,+)$ to $(\mathbb{N},\cdot)$. Let $X$ be a countable subset of $\mathbb{Q}$ that contains the image $\text{Im}(N(\mcm))$ with an extra condition: if $x_1,x_2 \in X$, the fraction $x_1/x_2$ belongs to $X$, too. By \cite[Theorem 2]{liu}, without loss of generality, one can assume $X = \mathbb{Q}$ or $X = \{ q^{z}: z \in \mathbb{Z} \}$ where $q$ is a power of a natural number strictly greater than 1. 

Given $\mcp$, $\mcm$, $X$, and for sufficiently large $x \in X$, we assume that the following condition holds: 
\begin{itemize}
    \item[$(\star)$] $I(x) := \sum\limits_{\substack{\mfm \in \mcm \\ N(\mfm) \leq x}} 1 = \kappa x + O(x^\theta)$, for some $\kappa > 0$ and $0 \leq \theta < 1$.
\end{itemize}

For each $\mfm \in \mcm$, we define
$$\omega(\mfm) = \sum_{\substack{\mfp \in \mathcal{P} \\ n_\mfp(\mfm) \geq 1}} 1,$$
the number of elements of $\mathcal{P}$ that generates $\mfm$, counted without multiplicity.


For a non-zero element $\mathfrak{m} \in \mathcal{M}$, let the prime element factorization of $\mathfrak{m}$ be given as
\begin{equation}\label{factorization}
\mathfrak{m} = s_1 \mfp_1 + \cdots + s_r \mfp_r,
\end{equation}
where $\mfp_i'$s are its distinct prime elements and $s_i'$s are their respective non-zero multiplicities. Here, $\omega(\mfm) = r$ and 
$$N(\mfm) = N(\mfp_1)^{s_1} \cdots N(\mfp_r)^{s_r}.$$ 

Let $h \geq 2$ be an integer. We say $\mathfrak{m}$ is an $h$\textit{-free} element if $s_i \leq h-1$ for all $i \in \{1, \cdots, r\}$, and we say $\mathfrak{m}$ is an $h$\textit{-full} element if $s_i \geq h$ for all $i \in \{1, \cdots, r\}$
. Let $\mathcal{S}_h$ denote the set of $h$-free elements and $\mathcal{N}_h$ denote the set of $h$-full elements. The distribution of $h$-free elements in $\mcm$ is well-established. To demonstrate the result, we introduce the generalized $\zeta$-function which is an analog of the classical Riemann $\zeta$-function as the following:
\[
\zeta_\mcm(s) := \sum_{\substack{\mathfrak{m}}} \frac1{(N(\mathfrak{m}) )^s} 
= \prod_{\mfp} \Big( 1- N (\mfp)^{-s}\Big)^{-1}
\ \text{for } \Re (s) >1,
\]
where $\mathfrak{m}$ and $\mfp$ respectively range through the non-zero elements in $\mcm$ and the prime elements in $\mcp$. Since Condition ($\star$) satisfies \cite[Chapter 4, Axiom A]{jk2} with $(G, \cdot) = (\mathcal{M}, +)$, $A = \kappa$, $\delta = 1$, and $\eta = \theta$, the absolute convergence of the above series for $\Re (s) >1$ follows from \cite[Chapter 4, Proposition 2.6]{jk2}. In the region $\Re (s) >1$, the proposition also proves the equality between the sum and the product, proves that $\zeta_\mcm(s) \neq 0$, and establishes the integral form of $\zeta_\mcm(s)$ as the following:
$$\zeta_\mcm(s) = s \int_1^\infty I(x) x^{-s-1} \ dx.$$

Let $\mathcal{P}, \ \mcm$, and $X$ satisfy the Condition ($\star$). Let $x \in X$ and let $\mathcal{S}_h(x)$ denote the set of $h$-free elements with norm $N(\cdot)$ less than or equal to $x$. Since Condition ($\star$) satisfies \cite[Chapter 4, Axiom A]{jk2}, thus, by \cite[Chapter 4, Proposition 5.5]{jk2}, we have:
\begin{equation}\label{hfreeidealcount}
    |\mathcal{S}_h(x)| = \frac{\kappa}{\zeta_\mcm(h)} x + O_h \big( R_{\mathcal{S}_h}(x) \big),
\end{equation}
where $|S|$ denotes the cardinality of the set $S$, and where
   \begin{equation}\label{RSh(x)}
    R_{\mathcal{S}_h}(x) = \begin{cases}
    x^\theta  & \text{ if } \frac{1}{h} < \theta, \\
    x^\frac{1}{h} (\log x) & \text{ if } \frac{1}{h} = \theta, \\
    x^{\frac{1}{h}} & \text{ if } \frac{1}{h} > \theta.
\end{cases}
    \end{equation}
As an addition to the literature, we prove the distribution of $h$-full elements. To do this, we need to define a new constant. Note that for a real sequence $\{ a_n \}$, the absolute convergence of the product $\prod_{n=1}^\infty (1 + a_n)$ is equivalent to the absolute convergence of the sum $\sum_{n=1}^\infty a_n$. Since $\sum_\mfp N(\mfp)^{-(1+1/h)}$ converges by Part 3 of \lmaref{boundnm}, thus the product 
\begin{equation}\label{gammahk}
    \gamma_{\scaleto{h}{4.5pt}} = \gamma_{\scaleto{h,\mcm}{5.5pt}} := \prod_\mfp \left( 1 + \frac{N(\mfp) - N(\mfp)^{1/h}}{N(\mfp)^2 \left( N(\mfp)^{1/h} - 1 \right)}\right)
\end{equation}
converges as well. The lemma also proves the convergence of all series of the form $\sum_{\mfp} N(\mfp)^{-\alpha}$ where $\alpha > 1$. We will use this reasoning without repetition to define other constants in our work.

For the distribution of $h$-full elements, we prove:
\begin{thm}\label{hfullideals}
Let $\mathcal{P}, \mcm$, and $X$ satisfy the Condition ($\star$). Let $x \in X$ and $h \geq 2$ be any integer. Let $\mathcal{N}_h(x)$ denote the set of $h$-full elements with norm $N(\cdot)$ less than or equal to $x$. We have
$$|\mathcal{N}_h(x)| = \kappa \gamma_{\scaleto{h}{4.5pt}} x^{1/h} + O_h \big( R_{\mathcal{N}_h}(x) \big),$$
where $\gamma_{\scaleto{h}{4.5pt}}$ is the constant defined in \eqref{gammahk},
and where
\begin{equation}\label{E2(x)}
    R_{\mathcal{N}_h}(x) = \begin{cases}
        x^{\frac{\theta}{h}} & \text{ if } \frac{h}{h+1} < \theta, \\
        x^{\frac{1}{h+1}} (\log x) & \text{ if } \frac{h}{h+i} = \theta  \text{ for some } i \in \{ 1, \ldots, h-1 \}, \\
        x^{\frac{1}{h+1}} & \text{ if } \frac{h}{h+1} > \theta \text{ and } \frac{h}{h+i} \neq \theta \text{ for any } i \in \{ 1, \cdots, h-1 \}.
    \end{cases} 
\end{equation}
\end{thm}
We use the above distribution results to study the distribution of $\omega(\mfm)$ over $h$-free and $h$-full elements. We begin with some definitions.

Let $x \in X$. Let $\mathfrak{A}$ and $\mathfrak{B}$ be constants defined as
\begin{equation}\label{A}
    \mfa:= \lim_{x \rightarrow \infty} \left( \sum_{\substack{\mfp \in \mathcal{P} \\ N(\mfp) \leq x}} \frac{1}{N(\mfp)} - \log \log x \right),
\end{equation}
and
\begin{equation}\label{B}
    \mathfrak{B} := \begin{cases}
        - \pi^2/6  & \text{if } X = \mathbb{Q} \text{ and} \\
         (\log \log q)^2 - \pi^2/6 & \text{if } X = \{ q^{z} : z \in \mathbb{Z} \}.
    \end{cases}
\end{equation}
The existence of the constant $\mathfrak{A}$ is explained in \cite[Lemma 2]{liuturan}. We define the constants
\begin{equation}\label{C1}
    \mathfrak{C}_1 :=  \mfa - \sum_{\mfp} \frac{N(\mfp)-1}{N(\mfp)(N(\mfp)^h -1)},
\end{equation}
and
\begin{equation}\label{C2}
    \mfc_2 := \mfc_1^2 + \mfc_1 + \mathfrak{B} - \sum_\mfp  \left( \frac{N(\mfp)^{h-1} - 1}{N(\mfp)^h - 1} \right)^2.
\end{equation}
For the distribution of $\omega(\mfm)$ over $h$-free elements, we prove:
\begin{thm}\label{hfreeomega}
Let $\mathcal{P}, \mcm$, and $X$ satisfy the Condition ($\star$). Let $x \in X$ and $h \geq 2$ be an integer. Let $\mathcal{S}_h(x)$ be the set of $h$-free elements with norm $N(\cdot)$ less than or equal to $x$. Then, we have
$$\sum_{\mathfrak{m} \in \mathcal{S}_h(x)} \omega(\mfm) = \frac{\kappa}{\zeta_\mcm(h)} x \log \log x + 
\frac{\kappa \mfc_1}{\zeta_\mcm(h)} x
+ O_h \left( \frac{x}{\log x}\right),$$
and
\begin{align*}
\sum_{\mathfrak{m} \in \mathcal{S}_h(x)} \omega^2(\mfm)
& = \frac{\kappa}{\zeta_\mcm(h)}  x (\log \log x)^2 + \frac{\kappa (2 \mfc_1 + 1)}{\zeta_\mcm(h)} x \log \log x + \frac{\kappa \mfc_2}{\zeta_\mcm(h)} x \\
& + O_h \left( \frac{x \log \log x}{\log x}\right).
\end{align*}
\end{thm} 
Next, we consider the case of $h$-full elements. Let $\mathfrak{L}_h(r)$ be the convergent sum defined for $r > h$ as
\begin{equation}\label{lhr}
     \mathfrak{L}_h(r) := \sum_\mfp \frac{1}{N(\mfp)^{(r/h)-1} \left( N(\mfp) - N(\mfp)^{1-1/h} + 1 \right)}.
\end{equation}
We also define two new constants
\begin{equation}\label{D1}
    \mathfrak{D}_1 :=  \mathfrak{A} - \log h + \mathfrak{L}_h(h+1) - \mathfrak{L}_h(2h),
\end{equation}
and
\begin{equation}\label{D2}
    \mathfrak{D}_2 := \mathfrak{D}_1^2 + \mathfrak{D}_1 + \mathfrak{B}
     - \sum_{\mfp} \left( \frac{1}{N(\mfp)-N(\mfp)^{1-1/h}+1} \right)^2.
\end{equation}
For the distributions of $\omega(\mfm)$ over $h$-full elements, we prove:
\begin{thm}\label{hfullomega}
Let $\mathcal{P}, \mcm$, and $X$ satisfy the Condition ($\star$). Let $x \in X$ and $h \geq 2$ be an integer. Let $\mathcal{N}_h(x)$ be the set of $h$-full elements with norm $N(\cdot)$ less than or equal to $x$. Let $\gamma_h$ be defined in \eqref{gammahk}. Then, we have
\begin{align*}
\sum_{\mfm \in \mathcal{N}_h(x)} \omega(\mfm) & = \kappa \gamma_{h} x^{1/h} \log \log x + \kappa \gamma_h \mathfrak{D}_1 x^{1/h} + O_h \left( \frac{x^{1/h}}{\log x} \right),
\end{align*}
and
\begin{align*}
     \sum_{\mfm \in \mathcal{N}_h(x)} \omega^2(\mfm)
     & = \kappa \gamma_{h} x^{1/h} (\log \log x)^2 + \kappa \gamma_h \left( 2 \mathfrak{D}_1 + 1 \right) x^{1/h} \log \log x + \kappa \gamma_h \mathfrak{D}_2 x^{1/h} \\
     & \hspace{.5cm} + O_h \left( \frac{x^{1/h} \log \log x}{\log x} \right).
\end{align*}
\end{thm}

In \cite{hardyram}, Hardy and Ramanujan strengthened the classical distribution result for $\omega(n)$ by proving that $\omega(n)$ has the normal order $\log \log n$ over natural numbers. In \cite[Section 22.11]{hw}, one can find another proof of this result using the variance of $\omega(n)$ (see \cite[(22.11.7)]{hw}). In \cite[Corollary 1]{liuturan}, using this method of variance, the third author provides an analog of this result over any abelian monoid $\mathcal{M}$. More precisely, she showed that for any $\epsilon > 0$, the number of elements of $\mcm$ with norm $N(\cdot)$ less than or equal to $x$ that do not satisfy the inequality
 $$(1-\epsilon) \log \log N(\mathfrak{m}) \leq \omega(\mathfrak{m}) \leq (1+ \epsilon) \log \log N(\mathfrak{m})$$
is $o(x)$ as $x \rightarrow \infty$. This shows that $\omega(\mathfrak{m})$ has normal order $\log \log N(\mathfrak{m})$ over $\mathcal{M}$.

Note that $\omega(n)$ may exhibit different normal orders over different subsets of $\mathbb{N}$. For instance, $\omega(p^r) = 1$ for any prime $p$ and any positive integer $r \geq 1$, and thus $\omega(n)$ has normal order 1 over the set of all prime powers. However, using the first two moments of the $\omega(n)$ over $h$-free and $h$-full numbers, the authors in \cite{dkl2} established that $\omega(n)$ has normal order $\log \log n$ over these restricted subsets as well. In this work, we generalize this to the abelian monoid case. First, we define the normal order of a function over a subset of $\mathcal{M}$. Let $\mathfrak{S} \subseteq \mathcal{M}$ and $\mathfrak{S}(x)$ denote the set of elements belonging to $\mathfrak{S}$ and with norm $N(\cdot)$ less than or equal to $x$. Let $|\mathfrak{S}(x)|$ denote the cardinality of $\mathfrak{S}(x)$. Let $f, F : \mathfrak{S} \rightarrow \mathbb{R}_{\geq 0}$ be two functions such that $F$ is non-decreasing, i.e., if $\mathfrak{m}_1, \mathfrak{m}_2 \in \mathcal{M}$ with $N(\mfm_1) \leq N(\mfm_2)$, then $F(\mfm_1) \leq F(\mfm_2)$. Then, $f(\mathfrak{m})$ is said to have normal order $F(\mathfrak{m})$ over $\mathfrak{S}$ if for any $\epsilon > 0$, the number of $\mathfrak{m} \in \mathfrak{S}(x)$ that do not satisfy the inequality
$$(1-\epsilon) F(\mathfrak{m}) \leq f(\mathfrak{m}) \leq (1+ \epsilon) F(\mathfrak{m})$$
is $o(|\mathfrak{S}(x)|)$ as $x \rightarrow \infty$. 

The set of $h$-free elements has a positive density $1/\zeta_\mcm(h)$ in $\mathcal{M}$. Thus, the proof of normal order of $\omega(\mathfrak{m})$ over $h$-free elements being $\log \log N(\mathfrak{m})$ follows from the classical case. In particular, one can establish the following:
\begin{cor}
For any $\epsilon > 0$, the number of $\mathfrak{m} \in \mathcal{S}_h(x)$ that do not satisfy the inequality
$$(1-\epsilon) \log \log N(\mathfrak{m}) \leq \omega(\mathfrak{m}) \leq (1+ \epsilon) \log \log N(\mathfrak{m})$$
is $o(|\mathcal{S}_h(x)|)$ as $x \rightarrow \infty$.
\end{cor}
On the other hand, the set of $h$-full elements has density zero in $\mathcal{M}$ and thus does not follow directly as the previous two result. However, writing an $h$-full element $\mathfrak{m}$ as $\mathfrak{m} = \mathfrak{r}_0^h \mathfrak{r}_1$ where $\mathfrak{r}_0$ is the product of all distinct prime element factors $\mfp$ of $\mathfrak{m}$ with $n_\mfp(\mfm) = h$, one can use the classical result on $\omega(\mathfrak{r}_0)$ to establish the normal order of $\omega(\mathfrak{m})$ over $h$-full elements. Additionally, working similarly to \cite[Proof of Theorem 1.3]{dkl2}, using the method of variance, we can establish the following result: 
\begin{cor}
For any $\epsilon > 0$, the number of $\mathfrak{m} \in \mathcal{N}_h(x)$ that do not satisfy the inequality
$$(1-\epsilon) \log \log N(\mathfrak{m}) \leq \omega(\mathfrak{m}) \leq (1+ \epsilon) \log \log N(\mathfrak{m})$$
is $o(|\mathcal{N}_h(x)|)$ as $x \rightarrow \infty$.
\end{cor}

Finally, our work can be extended to analogs of other arithmetic functions as well. Let $\Omega(\mathfrak{m})$ denote the number of prime element factors of $\mathfrak{m}$ counted with multiplicity. In particular, for the representation of $\mathfrak{m}$ given in \eqref{factorization}, $\Omega(\mathfrak{m}) = \sum_{i=1}^r s_i$. Using the methods employed for $\omega(\mathfrak{m})$ in this manuscript, we can establish analogous results for $\Omega(\mathfrak{m})$. In particular, we can establish the first and the second moments of $\Omega(\mathfrak{m})$ over $h$-free and over $h$-full elements. We can also prove that $\Omega(\mathfrak{m})$ has normal order $\log \log N(\mathfrak{m})$ over $\mathcal{M}$ and over $\mathcal{S}_h$, and has normal order $h \log \log N(\mathfrak{m})$ over $\mathcal{N}_h$.
\section{Lemmata}
In this section, we list several lemmas required for our study.
\begin{lma}
Let $\mathcal{P}, \mcm$, and $X$ satisfy the Condition ($\star$). Let $x \in X$. Then
\begin{equation}\label{gpnt}
    \Pi(x) := \sum\limits_{\substack{\mfp \in \mcp \\ N(\mfp) \leq x}} 1 = O \left(\frac{x}{\log x} \right).
\end{equation}
\end{lma}
\begin{proof}
    The result can be deduced from Condition ($\star$). The case of $X =\mathbb{Q}$ is a result of Landau \cite[Pages 665-670]{landaupnt}, and the case of $X = \{ q^z : z \in \mathbb{Z} \}$ is given by Knopfmacher \cite[Page 76, Theorem 8.3]{jk}.
\end{proof}
\begin{lma}\label{boundnm}
   Let $\mathcal{P}, \mcm$, and $X$ satisfy the Condition ($\star$). Let $x \in X$ and $\alpha$ be a real number. We have
   \begin{enumerate}
       \item If $0 \leq \alpha < 1$, 
       $$\sum_{\substack{\mfp \in \mathcal{P} \\ N(\mfp) \leq x}} \frac{1}{N(\mfp)^\alpha} = O_\alpha \left( \frac{x^{1- \alpha}}{\log x} \right).$$
       \item If $0 \leq \alpha < 1$, 
       $$\sum_{\substack{\mathfrak{m} \in  \mathcal{M} \backslash \{ 0 \} \\ N(\mathfrak{m}) \leq x}} \frac{1}{N(\mathfrak{m})^\alpha} = O_\alpha \left( x^{1- \alpha} \right).$$
       \item If $\alpha > 1$, then
       $$\sum_{\substack{\mathfrak{m} \in  \mathcal{M} \backslash \{ 0 \} \\ N(\mathfrak{m}) \leq x}} \frac{1}{N(\mathfrak{m})^\alpha} = O_\alpha( 1),$$
       and thus
       $$\sum_{\substack{\mfp \in \mathcal{P} \\ N(\mfp) \leq x}} \frac{1}{N(\mfp)^\alpha} = O_\alpha (1).$$
       \item As a generalization of Mertens' theorem, we have
       $$\sum_{\substack{\mfp \in \mathcal{P} \\ N(\mfp) \leq x}} \frac{1}{N(\mfp)} = \log \log x + \mfa + O \left( \frac{1}{\log x} \right),$$
       where $\mfa$ some constant that depends only on $\mcp$.
       \item We have
       $$\sum_{\substack{\mathfrak{m} \in \mathcal{M} \backslash \{ 0 \} \\ N(\mathfrak{m}) \leq x}} \frac{1}{N(\mathfrak{m})} = \kappa \log x + \mathfrak{A}' + O \left( \frac{1}{x^{1 -\theta}} \right),$$     where
        $$\mathfrak{A}' = \kappa + \int_1^\infty \left( I(y) - \kappa y \right) y^{-2} dy,$$
        and where $I(y)$ is defined in Condition ($\star$).
       \item If $\alpha > 1$, then
       $$\sum_{\substack{\mfp \in \mathcal{P} \\ N(\mfp) \geq x}} \frac{1}{N(\mfp)^\alpha} = O \left( \frac{1}{(\alpha-1)x^{\alpha-1} (\log x)} \right).$$
       \item If $\alpha > 1$, then
       $$\sum_{\substack{\mathfrak{m} \in \mathcal{M} \\ N(\mathfrak{m}) \geq x}} \frac{1}{N(\mathfrak{m})^\alpha} = O \left( \frac{1}{(\alpha-1)x^{\alpha-1}} \right).$$
   \end{enumerate}
\end{lma}
\begin{proof}
    Notice that Conditions (A) and \eqref{gpnt} satisfy the Conditions (A) and (B) in \cite[Page 574]{liuturan} and Axiom A of \cite[Chapter 4]{jk2}. Thus Parts 1, 3, and 4 follow from \cite[Lemma 1 and 2]{liuturan} and Part 5 follows from \cite[Proposition 2.8(i)]{jk2}. Parts 2, 6, and 7 follow from the technique of partial summation (see \cite[Lemma 1.2]{aler}).
\end{proof}
\begin{lma}\label{sumplogp}
    Let $\mathcal{P}, \mcm$, and $X$ satisfy the Condition ($\star$). If $X = \mathbb{Q}$, we have
    $$\sum_{N(\mfp) \leq x/2} \frac{1}{N(\mfp) \log (x/N(\mfp))} = O \left( \frac{\log \log x}{\log x} \right),$$
    and if $X = \{ q^{z} : z \in \mathbb{Z} \}$, we have
    $$\sum_{N(\mfp) \leq x/q} \frac{1}{N(\mfp) \log (x/N(\mfp))} = O \left( \frac{\log \log x}{\log x} \right).$$
\end{lma}
\begin{proof}
    Follows from \cite[Lemma 4]{liuturan}.
\end{proof}
\begin{lma}\label{sumnwpx/2} 
  Let $\mathcal{P}, \mcm$, and $X$ satisfy the Condition ($\star$). If $X = \mathbb{Q}$, we have
    $$\sum_{\substack{\mfp \\ N(\mfp) \leq x/2}}  \frac{1}{N(\mfp)} \log \log \frac{x}{N(\mfp)} = (\log \log x)^2 + \mfa \log \log x + \mathfrak{B} +  O \left( \frac{\log \log x}{\log x} \right),$$
and if $X = \{ q^{z} : z \in \mathbb{Z} \}$, we have
  $$\sum_{\substack{\mfp \\ N(\mfp) \leq x/q}}  \frac{1}{N(\mfp)} \log \log \frac{x}{N(\mfp)} = (\log \log x)^2 + \mfa \log \log x + \mathfrak{B} +  O \left( \frac{\log \log x}{\log x} \right),$$
where $\mfa$ and $\mathfrak{B}$ are defined in \eqref{A} and \eqref{B} respectively.
\end{lma}
\begin{proof}
    Follows from \cite[Proof of Lemma 3]{liuturan}.
\end{proof}
\begin{lma}\label{saidakeq}
  Let $\mathcal{P}, \mcm$, and $X$ satisfy the Condition ($\star$). Let $\mfp$ and $\mfq$ denote prime elements. Let $x \in X$. Then, we have
\begin{equation*}
    \sum_{\substack{\mfp, \mfq \\ N(\mfp) N(\mfq) \leq x}} \frac{1}{N(\mfp) N(\mfq)} = (\log \log x)^2 + 2 \mfa \log \log x + \mfa^2 + \mathfrak{B} + O \left( \frac{\log \log x}{\log x} \right),
\end{equation*}
where $\mfa$ and $\mathfrak{B}$ are defined in \eqref{A} and \eqref{B} respectively.
\end{lma}
\begin{proof}
   Let $X = \mathbb{Q}$. Since $N(\mfq) \geq 2$, thus $N(\mfp) N(\mfq) \leq x$ implies $N(\mfp) \leq x/2$. Thus, using  Part 4 of \lmaref{boundnm} and \lmaref{sumplogp}, we can write
   \begin{align*}
       & \sum_{\substack{\mfp, \mfq \\ N(\mfp) N(\mfq) \leq x}} \frac{1}{N(\mfp) N(\mfq)} \\
       & = \sum_{\substack{\mfp \\ N(\mfp) \leq x/2}} \frac{1}{N(\mfp)} \sum_{\substack{\mfq \\ N(\mfq) \leq x/N(\mfp)}} \frac{1}{N(\mfq)} \\
       & = \sum_{\substack{\mfp \\ N(\mfp) \leq x/2}}  \frac{1}{N(\mfp)} \left( \log \log \frac{x}{N(\mfp)} + \mfa + O \left( \frac{1}{\log \frac{x}{N(\mfp)}} \right) \right) \\
       & = \sum_{\substack{\mfp \\ N(\mfp) \leq x/2}}  \frac{1}{N(\mfp)} \log \log \frac{x}{N(\mfp)} + \mfa \left( \log \log \frac{x}{2} + \mfa \right) + O \left( \frac{\log \log x}{\log x} \right).
   \end{align*}
   Applying \lmaref{sumnwpx/2} to the above and the fact that
   \begin{equation}\label{loglogx/beta}
       \log \log \frac{x}{\beta} = \log \log x + O_\beta \left( \frac{1}{\log x} \right),
   \end{equation}
   for any constant $\beta \geq 2$, we complete the proof for this case. 
   
   Now, let $X = \{ q^{z} : z \in \mathbb{Z} \}$. Here, since $N(\mfq) \geq q$, thus $N(\mfp) N(\mfq) \leq x$ implies $N(\mfp) \leq x/q$. Thus, again using Part 4 of \lmaref{boundnm} and \lmaref{sumplogp}, we can write
   \begin{align*}
       & \sum_{\substack{\mfp, \mfq \\ N(\mfp) N(\mfq) \leq x}} \frac{1}{N(\mfp) N(\mfq)} \\
       & = \sum_{\substack{\mfp \\ N(\mfp) \leq x/q}} \frac{1}{N(\mfp)} \sum_{\substack{\mfq \\ N(\mfq) \leq x/N(\mfp)}} \frac{1}{N(\mfq)} \\
       & = \sum_{\substack{\mfp \\ N(\mfp) \leq x/q}}  \frac{1}{N(\mfp)} \left( \log \log \frac{x}{N(\mfp)} + \mfa + O \left( \frac{1}{\log \frac{x}{N(\mfp)}} \right) \right) \\
       & = \sum_{\substack{\mfp \\ N(\mfp) \leq x/q}}  \frac{1}{N(\mfp)} \log \log \frac{x}{N(\mfp)} + \mfa \left( \log \log \frac{x}{q} + \mfa \right) + O \left( \frac{\log \log x}{\log x} \right).
   \end{align*}
   Finally, we complete the proof using \lmaref{sumnwpx/2} and \eqref{loglogx/beta} to the above.
\end{proof}

In the following sections, we will always assume that $\mathcal{P}, \mcm$, and $X$ satisfy Condition ($\star$).
\section{The first and the second moment of \texorpdfstring{$\omega(\mathfrak{m})$}{} over \texorpdfstring{$h$}{}-free elements}
We begin this section by proving the distribution of $h$-free elements which are co-prime to a given set of prime elements. To do this, we introduce the generalized M\"obius $\mu$-function. For an $\mathfrak{m} \in \mathcal{M}$, the Mobius function $\mu_\mcm(\mathfrak{m})$ is defined as
$$\mu_\mcm(\mathfrak{m}) = \begin{cases}
    1 & \text{ if } \mathfrak{m} = 0, \\
    (-1)^r & \text{ if } \mathfrak{m} \text{ is $2$-free and generated by $r$ distinct prime elements},\\
    0 & \text{ if } \text{there exists a prime element $\mfp$ with $n_\mfp(\mfm) \geq 2$}.
\end{cases}
$$
Let $\mfm_1, \mfm_2 \in \mcm$. We say $\mfm_1$ divides $\mfm_2$ if the following holds: for any $\mfp \in \mcm$, if $n_\mfp(\mfm_1) = r$, then $n_{\mfp}(\mfm_2) \geq r$. We denote this property by $\mfm_1 | \mfm_2$. Note that, for any $\mathfrak{m}  \in \mathcal{M}$, $\mu_\mcm(\mathfrak{m})$ satisfies the identity
\begin{equation}\label{iden_mu}
    \sum_{\substack{\mathfrak{d} \in \mathcal{M} \\ \mathfrak{d}^h|\mathfrak{m}}} \mu_\mcm(\mathfrak{d}) = \begin{cases}
        1 & \text{ if $\mathfrak{m}$ is $h$-free}, \\
        0 & \text{ otherwise}.
    \end{cases}
\end{equation}
Note that the generating series for $\mu_\mcm(\mathfrak{m})$ is the reciprocal of the generalized $\zeta$-function. In particular, for $\Re(s) > 1$, we have
\begin{equation}\label{dedekind}
    \frac{1}{\zeta_\mcm(s)} = \sum_{\substack{\mathfrak{m} \in \mathcal{M} \backslash \{ 0 \} }} \frac{\mu_\mcm(\mathfrak{m})}{N(\mathfrak{m})^s} = \prod_{\mfp \in \mathcal{P}} \left( 1 - \frac{1}{N(\mfp)^s} \right).
\end{equation}
Note that the above equalities follow from the absolute convergence of the sum and the product formulas of $\zeta_\mcm(s)$ in the region $\Re(s) > 1$. We prove the following distribution result:
\begin{lma}\label{hfreeidealrestrict}
    Let $h \geq 2$ and $r \geq 1$ be integers. Let $\ell_1,\ldots,\ell_r$ be fixed distinct prime elements and $\mathcal{S}_{h,\ell_1,\ldots,\ell_r}(x)$ denote the set of $h$-free elements $\mfm$ with norm $N(\mfm) \leq x$ and with $n_{\ell_i}(\mfm) = 0$ for all $i \in \{1,\cdots, r\}$. Then, we have
    $$|\mathcal{S}_{h,\ell_1,\ldots,\ell_r}(x)| = \prod_{i=1}^r \left( \frac{N(\ell_i)^h - N(\ell_i)^{h-1}}{N(\ell_i)^h - 1} \right) \frac{\kappa}{\zeta_\mcm(h)} x + O_{h,r} \left( R_{\mathcal{S}_h}(x) \right),$$
    where $R_{\mathcal{S}_h}(x)$ is defined in \eqref{RSh(x)}.
\end{lma}
\begin{proof}
    We define the restricted generalized M\"obius $\mu$-function on $\mathcal{M}$ as
    $$\mu_{\mcm,\ell_1,\cdots,\ell_r}(\mathfrak{m}) = \begin{cases}
        \mu_\mcm(\mathfrak{m}) & \text{ if } n_{\ell_i}(\mfm) = 0 \ \forall i \in \{1, \ldots, r \}, \\
        0 & \text{ otherwise}.
    \end{cases}$$ 
    Using this and \eqref{iden_mu}, notice that
    \begin{align*}
        |\mathcal{S}_{h,\ell_1,\cdots,\ell_r}(x)| & = \sum_{\substack{\mathfrak{m} \in \mathcal{M} \\ N(\mathfrak{m}) \leq x \\ n_{\ell_i}(\mfm) = 0, \ \forall i }} \sum_{\substack{\mathfrak{d} \in\mathcal{M} \\ \mathfrak{d}^h|\mfm \\ n_{\ell_i}(\mfd) = 0, \ \forall i}} \mu_\mcm(\mathfrak{d}) \\
        & = \sum_{\substack{\mathfrak{m} \in \mathcal{M} \\ N(\mathfrak{m}) \leq x \\ n_{\ell_i}(\mfm) = 0, \ \forall i }} \sum_{\substack{\mathfrak{d} \in\mathcal{M} \\ \mathfrak{d}^h|\mfm}} \mu_{\mcm,\ell_1,\cdots,\ell_r}(\mathfrak{d}) \\
        & = \sum_{\substack{\mathfrak{d} \in\mathcal{M} \\ N(\mathfrak{d})^h \leq x}} \mu_{\mcm,\ell_1,\cdots,\ell_r}(\mathfrak{d}) \sum_{\substack{\mathfrak{m} \in \mathcal{M} \\ N(\mathfrak{m}) \leq x/N(\mathfrak{d})^h \\ n_{\ell_i}(\mfm) = 0, \ \forall i}} 1.
    \end{align*}
    We estimate the inner sum above using inclusion-exclusion, 
    Condition ($\star$), and the fact that $N(\ell_i) \geq 2$ for all $i$ as
    \begin{align*}
        & \sum_{\substack{\mathfrak{m} \in \mathcal{M} \\ N(\mathfrak{m}) \leq x/N(\mathfrak{d})^h \\ n_{\ell_i}(\mfm) = 0, \ \forall i}} 1 \\
        & = I \left( \frac{x}{N(\mathfrak{d})^h} \right) - \sum_{i=1}^r I \left( \frac{x}{N(\mathfrak{d})^h N(\ell_i)} \right) + \sum_{\substack{i,j=1 \\ i \neq j}}^r I \left( \frac{x}{N(\mathfrak{d})^h N(\ell_i) N(\ell_j)} \right) - \\
        & \hspace{.5cm} + (-1)^r I \left( \frac{x}{N(\mathfrak{d})^h N(\ell_1) \cdots N(\ell_r)} \right) \\
        & =  \frac{\kappa x}{N(\mathfrak{d})^h} +  \frac{\kappa x}{N(\mathfrak{d})^h} \sum_{i=1}^r \left( \frac{1}{N(\ell_i)} \right) + \frac{\kappa x}{N(\mathfrak{d})^h} \sum_{\substack{i,j=1 \\ i \neq j}}^r \left( \frac{1}{N(\ell_i) N(\ell_j)} \right) \\
        & \hspace{.5cm} + \frac{(-1)^r \kappa x}{N(\mathfrak{d})^h} \left( \frac{1}{N(\ell_1) \cdots N(\ell_r)} \right) + O \left( 2^r \frac{x^\theta}{N(\mathfrak{d})^{h \theta}} \right) \\
        & = \frac{\kappa x}{N(\mathfrak{d})^h} \prod_{i=1}^r  \left( 1 - \frac{1}{N(\ell_i)} \right) + O \left( 2^r \frac{x^\theta }{N(\mathfrak{d})^{h \theta}} \right),
    \end{align*}
    where the factor $2^r$ bounds the number of possible subsets of $\{ \ell_1, \cdots, \ell_r \}$. Combining the last two results, we obtain
    \begin{align*}
        |\mathcal{S}_{h,\ell_1,\cdots,\ell_r}(x)| & = \kappa x \prod_{i=1}^r \left( 1 - \frac{1}{N(\ell_i)} \right) \sum_{\substack{\mathfrak{d} \in\mathcal{M} \\ N(\mathfrak{d})^h \leq x}} \frac{\mu_{\mcm,\ell_1,\cdots,\ell_r}(\mathfrak{d})}{N(\mathfrak{d})^h} + O \left( 2^r \sum_{\substack{\mathfrak{d} \in\mathcal{M} \\ N(\mathfrak{d})^h \leq x}} \frac{x^\theta }{N(\mathfrak{d})^{h \theta}} \right) \\
        & = \kappa x \prod_{i=1}^r \left( 1 - \frac{1}{N(\ell_i)} \right)  \sum_{\substack{\mathfrak{d} \in\mathcal{M} \backslash \{ 0 \}}} \frac{\mu_{\mcm,\ell_1,\cdots,\ell_r}(\mathfrak{d})}{N(\mathfrak{d})^h} + O \left( x \sum_{\substack{\mathfrak{d} \in\mathcal{M} \\ N(\mathfrak{d}) > x^{1/h}}} \frac{1}{N(\mathfrak{d})^h} \right) \\
        & \hspace{.5cm} + O \left( 2^r \sum_{\substack{\mathfrak{d} \in\mathcal{M} \\ N(\mathfrak{d}) \leq x^{1/h}}} \frac{x^\theta }{N(\mathfrak{d})^{h \theta}} \right).
    \end{align*}
Note that, using the Euler product formula mentioned in \eqref{dedekind}, we can write the main term on the right side above as
$$\sum_{\substack{\mathfrak{d} \in \mathcal{M} \backslash \{ 0 \} }} \frac{\mu_{\mcm,\ell_1,\cdots,\ell_r}(\mathfrak{d})}{N(\mathfrak{d})^h}  
= \prod_{\substack{\mfp \in \mathcal{P} \\ \mfp \neq \ell_1, \cdots, \ell_r}} \left( 1 - \frac{1}{N(\mfp)^h} \right) = \frac{1}{\zeta_\mcm(h)} \left( \prod_{i=1}^r \frac{N(\ell_i)^h}{N(\ell_i)^h - 1} \right).$$
We bound the first error term using Part 7 of \lmaref{boundnm} with $\alpha = h$ as
\begin{equation*}
    \sum_{\substack{\mathfrak{d} \in\mathcal{M} \\ N(\mathfrak{d}) > x^{1/h}}} \frac{1}{N(\mathfrak{d})^h} \ll \frac{1}{x^{1 - \frac{1}{h}}}.
\end{equation*}
Also, for the second error term, using Parts 2, 3, and 5 of \lmaref{boundnm}, we obtain
\begin{equation*}
    \sum_{\substack{\mathfrak{d} \in\mathcal{M} \\ N(\mathfrak{d}) \leq x^{1/h}}} \frac{1}{N(\mathfrak{d})^{h \theta}} \ll_h \begin{cases}
    1 & \text{ if } \frac{1}{h} < \theta, \\
    \log x & \text{ if } \frac{1}{h} = \theta, \text{ and} \\
    x^{\frac{1}{h} - \theta} & \text{ if } \frac{1}{h} > \theta.
\end{cases}
\end{equation*}
Combining the above four results, we complete the proof.
\end{proof}

Finally, we establish the asymptotic distribution of $\omega(\mathfrak{m})$ over $h$-free elements:
\begin{proof}[\textbf{Proof of \thmref{hfreeomega}}]
Let $\mfp^k|| \mathfrak{m}$ denote the property that $n_\mfp(\mfm) = k$. Then writing $\mathfrak{m} = k \mfp + y$ with $n_\mfp(y)=0$ for such $\mathfrak{m}$ and using $m = \min \left\{ h-1, \left\lfloor \log x / \log N(\mfp) \right\rfloor \right\}$, we obtain
\begin{equation*}\label{anomega1}
    \sum_{\mathfrak{m} \in \mathcal{S}_h(x)} \omega(\mathfrak{m}) = \sum_{\mathfrak{m} \in \mathcal{S}_h(x)} \sum_{\substack{\mfp \\ n_\mfp(\mfm) \geq 1}} 1 = \sum_{N(\mfp) \leq x} \sum_{k=1}^m \sum_{\substack{\mathfrak{m} \in \mathcal{S}_h(x) \\ \mfp^k || \mathfrak{m}}} 1 = \sum_{N(\mfp) \leq x} \sum_{k=1}^m \sum_{\substack{y \in \mathcal{S}_h(x/N(\mfp)^k) \\ n_\mfp(y) =0}} 1.
\end{equation*}
Now, using \lmaref{hfreeidealrestrict} for one prime, writing $R_{\mathcal{S}_h}(x) \ll x^\nu$ where $0 \leq \nu < 1$ and using Part 1 of \lmaref{boundnm} to the above, we obtain
\begin{align}\label{anomega2}
    \sum_{\mathfrak{m} \in \mathcal{S}_h(x)} \omega(\mathfrak{m}) 
& = \sum_{N(\mfp) \leq x} \sum_{k=1}^m \left( \frac{N(\mfp)^{h} - N(\mfp)^{h-1}}{N(\mfp)^k(N(\mfp)^h - 1)} \right) \frac{x}{\zeta_\mcm(h)} + O_h \left( x^\nu \sum_{N(\mfp) \leq x} 
  \frac{1}{N(\mfp)^{\nu}} \right) \notag \\
 & = \sum_{N(\mfp) \leq x} \sum_{k=1}^m \left( \frac{N(\mfp)^{h} - N(\mfp)^{h-1}}{N(\mfp)^k(N(\mfp)^h - 1)} \right) \frac{\kappa}{\zeta_\mcm(h)} x + O_h \left(  \frac{x}{\log x} \right).
\end{align}
Using $\eqref{gpnt}$, we write
\begin{align}\label{caser=h-1}
    \sum_{N(\mfp) \leq x} \sum_{k=1}^m \left( \frac{N(\mfp)^{h} - N(\mfp)^{h-1}}{N(\mfp)^k(N(\mfp)^h - 1)} \right)
    & = \sum_{N(\mfp) \leq x} \sum_{k=1}^{h-1} \left( \frac{N(\mfp)^{h} - N(\mfp)^{h-1}}{N(\mfp)^k(N(\mfp)^h - 1)} \right) + O \left( \frac{1}{\log x} \right).
\end{align}
Now, using 
$$\sum_{k=1}^{h-1} \frac{N(\mfp)^{h} - N(\mfp)^{h-1}}{N(\mfp)^k(N(\mfp)^h - 1)}  =  \frac{1}{N(\mfp)} - \frac{N(\mfp) - 1}{N(\mfp)(N(\mfp)^h - 1)},$$ 
and Parts 4 and 6 of \lmaref{boundnm}, we obtain  
\begin{align*}
    \sum_{N(\mfp) \leq x} \sum_{k=1}^{h-1} \left( \frac{N(\mfp)^{h} - N(\mfp)^{h-1}}{N(\mfp)^k(N(\mfp)^h - 1)} \right)
    & = \log \log x + \mfc_1 + O_h \left( \frac{1}{\log x} \right),
\end{align*}
where $\mfc_1$ is defined in \eqref{C1}. Combining the above with \eqref{anomega2} and \eqref{caser=h-1} completes the first part of the proof.
Next, let $m_\mfp = \min \left\{ h-1, \left\lfloor \log x /\log N(\mfp) \right\rfloor \right\}$ and $m_\mathfrak{q} = \min \left\{ h-1, \left\lfloor \log x / \log N(\mathfrak{q}) \right\rfloor \right\}$. Then, we have
\begin{equation}\label{mainpart}
    \sum_{\mathfrak{m} \in \mathcal{S}_h(x)} \omega^2(\mathfrak{m}) = \sum_{\mathfrak{m} \in \mathcal{S}_h(x)} \left( \sum_{\substack{\mfp \\ n_\mfp(\mfm) \geq 1}} 1 \right)^2 = \sum_{\mathfrak{m} \in \mathcal{S}_h(x)} \omega(\mathfrak{m}) + \sum_{\mathfrak{m} \in \mathcal{S}_h(x)} \sum_{\substack{\mfp,\mathfrak{q} \\ \mfp^k || \mathfrak{m}, \ \mathfrak{q}^l || \mathfrak{m}, \ \mfp \neq \mathfrak{q} }} \left( \sum_{k=1}^{m_\mfp} \sum_{l=1}^{m_\mathfrak{q}} 1 \right),
\end{equation}
where $\mfp$ and $\mathfrak{q}$ above denote prime ideals. The first sum on the right-hand side is the first moment studied above. For the second sum, we rewrite the sum and use \lmaref{hfreeidealrestrict} for two primes $\mfp$ and $\mathfrak{q}$ and $R_{\mathcal{S}_h}(x) \ll x^\nu, \ 0 < \nu < 1,$ again to obtain
\begin{align}\label{mainpart2}
    &  \sum_{\mathfrak{m} \in \mathcal{S}_h(x)} \sum_{\substack{\mfp,\mathfrak{q} \\ \mfp^k || \mathfrak{m}, \ \mathfrak{q}^l || \mathfrak{m}, \ \mfp \neq \mathfrak{q} }} \left( \sum_{k=1}^{m_\mfp} \sum_{l=1}^{m_\mathfrak{q}} 1 \right) \notag \\
    & = \sum_{\substack{\mfp, \mathfrak{q} \\ \mfp \neq \mathfrak{q}, \ N(\mfp) N(\mathfrak{q}) \leq x}} \sum_{k=1}^{m_\mfp} \sum_{l=1}^{m_\mathfrak{q}} \Bigg( \left( \frac{N(\mfp)^{h} - N(\mfp)^{h-1}}{N(\mfp)^{k}(N(\mfp)^h - 1)} \right) \left( \frac{N(\mathfrak{q})^{h} - N(\mathfrak{q})^{h-1}}{N(\mathfrak{q})^{l}(N(\mathfrak{q})^h - 1)} \right) \frac{x}{\zeta(h)} \notag \\
    & \hspace{.5cm} + O_h \left( \frac{x^{\nu}}{N(\mfp)^{k \nu} N(\mathfrak{q})^{l \nu}} \right) \Bigg). 
\end{align}
We employ Part 1 of \lmaref{boundnm} and \lmaref{sumplogp} 
to estimate the error term above as 
\begin{equation}\label{mainpart3}
    \sum_{\substack{\mfp,\mathfrak{q} \\ \mfp \neq \mathfrak{q}, \ N(\mfp) N(\mathfrak{q}) \leq x}} \sum_{k=1}^{m_\mfp} \sum_{l=1}^{m_\mathfrak{q}} \frac{x^{\nu}}{N(\mfp)^{k \nu} N( \mathfrak{q})^{l \nu}} \ll x^{\nu} \sum_{\substack{\mfp,\mathfrak{q} \\ \mfp \neq \mathfrak{q}, \ N(\mfp) N(\mathfrak{q}) \leq x}} \frac{1}{N(\mfp)^\nu N(\mathfrak{q})^{\nu}} \ll_h \frac{x \log \log x}{\log x}.
\end{equation}
For estimating the main term in \eqref{mainpart2}, 
we consider the set $R$ defined as
\begin{equation*}
    R:= \left\{ N(\mfp) \leq x \ | \  \left\lfloor \frac{\log x}{\log N(\mfp)}  \right\rfloor < h-1 \right\}.
\end{equation*}
Using the definition of $m_\mfp$ and $m_\mfq$, we rewrite 
\begin{align}\label{parts}
    & \sum_{\substack{\mfp,\mathfrak{q} \\ \mfp \neq \mathfrak{q}, \ N(\mfp) N(\mathfrak{q}) \leq x}} \sum_{k=1}^{m_\mfp} \sum_{l=1}^{m_\mathfrak{q}} \left( \frac{N(\mfp^{h}) - N(\mfp^{h-1})}{N(\mfp^{k})(N(\mfp^h) - 1)} \right) \left( \frac{N(\mathfrak{q}^{h}) - N(\mathfrak{q}^{h-1})}{N(\mathfrak{q}^{l}) (N(\mathfrak{q}^h) - 1)} \right) \notag \\
    & = \sum_{\substack{\mfp,\mathfrak{q} \\ \mfp \neq \mathfrak{q}, \ N(\mfp) N(\mathfrak{q}) \leq x}} \sum_{k=1}^{h-1} \sum_{l=1}^{h-1} \left( \frac{N(\mfp^{h}) - N(\mfp^{h-1})}{N(\mfp^{k})(N(\mfp^h) - 1)} \right) \left( \frac{N(\mathfrak{q}^{h}) - N(\mathfrak{q}^{h-1})}{N(\mathfrak{q}^{l}) (N(\mathfrak{q}^h) - 1)} \right) -  I_1  -  I_2 + I_3,
\end{align}
where
$$I_1 = \sum_{\substack{\mfp,\mathfrak{q} \\ \mfp \neq \mathfrak{q}, \ N(\mfp) N(\mathfrak{q}) \leq x \\ \mfp \in R}} \sum_{k=\left\lfloor \frac{\log x}{\log N(\mfp)}  \right\rfloor + 1}^{h-1} \sum_{l=1}^{h-1} \left( \frac{N(\mfp^{h}) - N(\mfp^{h-1})}{N(\mfp^{k})(N(\mfp^h) - 1)} \right) \left( \frac{N(\mathfrak{q}^{h}) - N(\mathfrak{q}^{h-1})}{N(\mathfrak{q}^{l}) (N(\mathfrak{q}^h) - 1)} \right),$$
$$I_2 = \sum_{\substack{\mfp,\mathfrak{q} \\ \mfp \neq \mathfrak{q}, \ N(\mfp) N(\mathfrak{q}) \leq x \\ \mathfrak{q} \in R}} \sum_{k=1}^{h-1} \sum_{l=\left\lfloor \frac{\log x}{\log N(\mathfrak{q})}  \right\rfloor + 1}^{h-1} \left( \frac{N(\mfp^{h}) - N(\mfp^{h-1})}{N(\mfp^{k})(N(\mfp^h) - 1)} \right) \left( \frac{N(\mathfrak{q}^{h}) - N(\mathfrak{q}^{h-1})}{N(\mathfrak{q}^{l}) (N(\mathfrak{q}^h) - 1)} \right),$$
$$I_3 = \sum_{\substack{\mfp,\mathfrak{q} \\ \mfp \neq \mathfrak{q}, \ N(\mfp) N(\mathfrak{q}) \leq x \\ \mfp, \mathfrak{q} \in R}} \sum_{k=\left\lfloor \frac{\log x}{\log N(\mfp)}  \right\rfloor + 1}^{h-1} \sum_{l=\left\lfloor \frac{\log x}{\log N(\mfq)}  \right\rfloor + 1}^{h-1} \left( \frac{N(\mfp^{h}) - N(\mfp^{h-1})}{N(\mfp^{k})(N(\mfp^h) - 1)} \right) \left( \frac{N(\mathfrak{q}^{h}) - N(\mathfrak{q}^{h-1})}{N(\mathfrak{q}^{l}) (N(\mathfrak{q}^h) - 1)} \right).$$
The sum $I_1$, $I_2$, and $I_3$ contribute to the error term. Note that $I_1 = I_2$. In fact, using \eqref{gpnt} and \lmaref{sumplogp}, we estimate
$$I_1 \ll \sum_{\substack{\mfp,\mathfrak{q} \\ \mfp \neq \mathfrak{q}, \ N(\mfp) N(\mathfrak{q}) \leq x \\ \mfp \in R}} \frac{1}{N(\mfp)^{\frac{\log x}{\log N(\mfp)}}} \frac{1}{N(\mfq)} \ll \sum_{N(\mfq) \leq x/2} \frac{1}{N(\mfq)^2 \log(x/N(\mfq))} \ll \frac{\log \log x}{\log x}.$$
Similarly, for $I_3$, we have
$$I_3 \ll \sum_{\substack{\mfp,\mathfrak{q} \\ \mfp \neq \mathfrak{q}, \ N(\mfp) N(\mathfrak{q}) \leq x \\ \mfp, \mathfrak{q} \in R}} \frac{1}{N(\mfp)^{\frac{\log x}{\log N(\mfp)}}} \frac{1}{N(\mathfrak{q})^{\frac{\log x}{\log N(\mfq)}}} \ll \frac{1}{x} \sum_{N(\mfq) \leq x}  \frac{1}{N(\mfq) \log(x/N(\mfq))}  \ll \frac{\log \log x}{x \log x}.$$
We next estimate the main term in \eqref{parts}. First, we rewrite the sum as
\begin{align}\label{part1}
    & \sum_{\substack{\mfp,\mathfrak{q} \\ \mfp \neq \mathfrak{q}, \ N(\mfp) N(\mathfrak{q}) \leq x}} \sum_{k=1}^{h-1} \sum_{l=1}^{h-1} \left( \frac{N(\mfp^{h}) - N(\mfp^{h-1})}{N(\mfp^{k})(N(\mfp^h) - 1)} \right) \left( \frac{N(\mathfrak{q}^{h}) - N(\mathfrak{q}^{h-1})}{N(\mathfrak{q}^{l}) (N(\mathfrak{q}^h) - 1)} \right) \notag \\
    & = \sum_{\substack{\mfp, \mfq \\ N(\mfp) N(\mathfrak{q}) \leq x}} \left( \frac{N(\mfp^{h-1}) - 1}{N(\mfp^h) - 1} \right) \left( \frac{N(\mathfrak{q}^{h-1}) - 1}{N(\mathfrak{q}^h) - 1}\right) - \sum_{\substack{\mfp \\ N(\mfp) \leq x^{1/2}}} \left( \frac{N(\mfp^{h-1}) - 1}{N(\mfp^h) - 1} \right)^2.
\end{align}
The second sum above is estimated using Part 6 of \lmaref{boundnm} as
\begin{align}\label{part2}
    \sum_{\substack{\mfp \\ N(\mfp) \leq x^{1/2}}} \left( \frac{N(\mfp^{h-1}) - 1}{N(\mfp^h) - 1} \right)^2 
    & = \sum_{\substack{\mfp}} \left( \frac{N(\mfp^{h-1}) - 1}{N(\mfp^h) - 1} \right)^2 + O \left( \frac{1}{x^{1/2} \log x}\right).
\end{align}
For the first sum on the right-hand side in \eqref{part1}, using 
$$\frac{N(\mfp^{h-1}) - 1}{N(\mfp^h) -1} = \frac{1}{N(\mfp)} - \frac{N(\mfp)-1}{N(\mfp)(N(\mfp^h) - 1)},$$ and the symmetry in $\mfp$ and $\mfq$, we have
\begin{align*}
    & \sum_{\substack{\mfp,\mfq \\ N(\mfp) N(\mathfrak{q}) \leq x}} \left( \frac{N(\mfp^{h-1}) - 1}{N(\mfp^h) - 1} \right) \left( \frac{N(\mathfrak{q}^{h-1}) - 1}{N(\mathfrak{q}^h) - 1}\right) \notag \\
    & = \sum_{\substack{\mfp, \mfq \\ N(\mfp) N(\mathfrak{q}) \leq x}} \frac{1}{N(\mfp \mfq)} - 2 \sum_{\substack{\mfp, \mfq \\ N(\mfp) N(\mathfrak{q}) \leq x}} \frac{1}{N(\mfp)} \left( \frac{N(\mfq)-1}{N(\mfq) (N(\mathfrak{q}^h) - 1)} \right) \\
    & \hspace{.5cm} + \sum_{\substack{\mfp, \mfq \\ N(\mfp) N(\mathfrak{q}) \leq x}} \left( \frac{N(\mfp)-1}{N(\mfp)(N(\mfp^h) - 1)} \right) \left( \frac{N(\mfq)-1}{N(\mfq)(N(\mathfrak{q}^h) - 1)} \right).
\end{align*}
We estimate the sums on the right-hand side above separately. For the first sum, we use \lmaref{saidakeq}. Let $X =\mathbb{Q}$. For the second sum, we use Parts 4 and 6 of \lmaref{boundnm}, and $\eqref{gpnt}$ to obtain
\begin{align*}
    & \sum_{\substack{\mfp, \mfq \\ N(\mfp) N(\mathfrak{q}) \leq x}} \frac{1}{N(\mfp)} \left( \frac{N(\mfq)-1}{N(\mfq)(N(\mathfrak{q}^h) - 1)} \right) \\
    & = \sum_{\substack{\mfp \\ N(\mfp) \leq x/2}} \frac{1}{N(\mfp)} \left( \sum_{\mfp} \frac{N(\mfp)-1}{N(\mfp)(N(\mfp^h) - 1)}  + O \left( \frac{1}{(x/N(\mfp)) \log (x/N(\mfp))}\right)\right) \notag \\
    & = \left( \sum_{\mfp} \frac{N(\mfp)-1}{N(\mfp)(N(\mfp^h) - 1)} \right) \left( \log \log x + \mfa \right) + O \left( \frac{1}{\log x} \right).
\end{align*}
For $X = \{ q^{z} : z \in \mathbb{Z} \}$, the above results follow similarly, but with the replacement of $x/2$ with $x/q$. Similarly, for the third sum, we use Part 6 of \lmaref{boundnm}, and \lmaref{sumplogp} to obtain
\begin{align*}
     & \sum_{\substack{\mfp, \mfq \\ N(\mfp) N(\mathfrak{q}) \leq x}} \left( \frac{N(\mfp)-1}{N(\mfp)(N(\mfp^h) - 1)} \right) \left( \frac{N(\mfq)-1}{N(\mfq)(N(\mathfrak{q}^h) - 1)} \right) \\
    & = \left( \sum_{\mfp} \frac{N(\mfp)-1}{N(\mfp)(N(\mfp)^h - 1)} \right)^2 + O \left( \frac{\log \log x}{x \log x}\right) .
\end{align*}
Combining the last three results and \lmaref{saidakeq}, we obtain
\begin{align*}
    & \sum_{\substack{\mfp, \mfq \\ N(\mfp) N(\mathfrak{q}) \leq x}} \left( \frac{N(\mfp^{h-1}) - 1}{N(\mfp^h) - 1} \right) \left( \frac{N(\mathfrak{q}^{h-1}) - 1}{N(\mathfrak{q}^h) - 1}\right) \\
    & = (\log \log x)^2 + 2 \mfa \log \log x + \mfa^2 + \mathfrak{B} + O \left( \frac{\log \log x}{\log x} \right).
\end{align*}
Combining \eqref{mainpart}, \eqref{mainpart2}, \eqref{mainpart3}, \eqref{parts}, \eqref{part1}, and \eqref{part2} with the above equation and using the first-moment estimate from the first part, we obtain the required second moment.
\end{proof}

\section{The first and the second moment of \texorpdfstring{$\omega(\mathfrak{m})$}{} over \texorpdfstring{$h$}{}-full elements}
Let $h \geq 2$ be an integer. Recall that $\mathcal{N}_h$ denotes the set of $h$-full elements in $\mcm$. Let $\mathfrak{m} \in \mathcal{N}_h$. We can separate the prime elements in the prime element factorization of $\mathfrak{m}$ (see \eqref{factorization}) based on their multiplicity modulo $h$ as 
\begin{align}\label{hfullfactorize}
\mathfrak{m} & =  \underbrace{h t_{1,0} \mfp_{1,0} + \cdots + h t_{l_0,0} \mfp_{l_0,0}}_{\text{multiplicity } \equiv 0\pmod{h}} + \underbrace{(h t_{1,1} + 1) \mfp_{1,1} + \cdots + (h t_{l_1,1}+1) \mfp_{l_1,1}}_{\text{multiplicity } \equiv 1\pmod{h}} + \cdots \notag \\
& \hspace{.5cm} \cdots + \underbrace{(h t_{1,h-1} + h-1) \mfp_{1,h-1} + \cdots + (h t_{l_{h-1},h-1} + h-1) \mfp_{l_{h-1},h-1}}_{\text{multiplicity } \equiv h-1 \pmod{h}},
\end{align}
where $\mfp_{i,j}$ with $1 \leq i \leq l_j$ denote a distinct prime with its multiplicity congruent to $j$ modulo $h$ and $l_j$ denotes the number of such distinct primes. 

We set
\begin{align*}
\mathfrak{a}_0 & =  t_{1,0} \mfp_{1,0} + \cdots + t_{l_0,0} \mfp_{l_0,0} + (t_{1,1} - 1) \mfp_{1,1} + \cdots + (t_{l_1,1} - 1) \mfp_{l_1,1} + \cdots \notag \\
& \hspace{.5cm} \cdots + (t_{1,h-1} - 1) \mfp_{1,h-1} + \cdots + (t_{l_{h-1},h-1} -1) \mfp_{l_{h-1},h-1},
\end{align*}
and for $j = 1, \ldots, h-1$, we set
$$\mathfrak{a}_j = \mfp_{1,j} + \cdots + \mfp_{l_j,j}.$$
Note that $\mathfrak{a}_1,\ldots, \mathfrak{a}_{h-1}$ are 2-free and any two $\mathfrak{a}_j'$s do not share any primes in their prime element factorization. 

Thus, we can write
\begin{equation}\label{defnhfull}
\mathfrak{m} = h \mathfrak{a}_0 + (h+1) \mathfrak{a}_1 + \cdots + (2h-1) \mathfrak{a}_{h-1}.
\end{equation}
Notice that the generating series for the $h$-full ideals is defined on $\Re(s) > 1/h$ as:
\begin{align}\label{formulanhk}
    \mathfrak{N}_h(s) & := \sum_{\mathfrak{m} \in \mathcal{N}_h} \frac{1}{N(\mathfrak{m})^s} \notag \\
    & = \prod_\mfp \left( 1 + \frac{1}{N(\mfp)^{hs}} + \cdots + \frac{1}{N(\mfp)^{ks}} + \cdots \right) \notag \\
    & = \prod_\mfp \left( 1 +  \frac{N(\mfp)^{-hs}}{1 - N(\mfp)^{-s}} \right).
\end{align}
Let $x \in X$. Let $\mathcal{N}_h(x)$ denote the set of $h$-full elements with a norm not exceeding $x$. Then, by the definition of an $h$-full element \eqref{defnhfull}, we have
$$|\mathcal{N}_h(x)| = \sum_{\substack{\mathfrak{m} \in \mathcal{N}_h \\ N(\mathfrak{m}) \leq x}} 1 = \sum_{\substack{\mathfrak{a}_0,\cdots, \mathfrak{a}_{h-1} \in \mathcal{M} \\ N(\mathfrak{a}_0)^h N(\mathfrak{a}_1)^{h+1} \cdots N(\mathfrak{a}_{h-1})^{2h -1} \leq x}} \mu_\mcm^2(\mathfrak{a}_1 + \cdots + \mathfrak{a}_{h-1}),$$
where the inner sum is non-zero if and only if $\mathfrak{a}_1,\ldots, \mathfrak{a}_{h-1}$ are 2-free and any two of them do not share a prime element in their factorization. By \cite[(1.5)]{is}, we have that there exists constants $\alpha_{r,h}$ where $2h+2 < r \leq (3h^2+h-2)/2$ such that the following equality is satisfied
\begin{equation}\label{iden2}
    \left( 1 + \frac{v^h}{1-v} \right) (1-v^h) (1 -v^{h+1}) \cdots (1-v^{2h-1}) = 1 - v^{2h+2} + \sum_{r=2h+3}^{(3h^2+h-2)/2} \alpha_{r,h} v^r.
\end{equation}
Note that $|\alpha_{r,h}| \leq h 2^h$. Now, substituting $v= N(\mfp)^{-s}$ above, taking the product over all prime ideals, and using the product formula for the generalized $\zeta$-function \eqref{dedekind}, we obtain
\begin{equation}\label{Nhxformula}
\mathfrak{N}_h(s) = \zeta_\mcm(hs) \zeta_\mcm((h+1)s) \cdots \zeta_\mcm((2h-1)s) \zeta_\mcm^{-1}((2h+2)s) \phi_h^\mcm(s),
\end{equation}
where $\phi_h^\mcm(s)$ satisfies
\begin{equation}\label{iden1}
\prod_{\mfp} \left( 1 - N(\mfp)^{-(2h+2)s} + \sum_{r = 2h+3}^{(3h^2+h-2)/2} \alpha_{r,h} N(\mfp)^{-rs} \right) = \zeta_\mcm^{-1}((2h+2)s) \phi_h^\mcm(s).
\end{equation}
For $h =2$, the sum on the right side of \eqref{iden2} is empty, and thus, $\phi_2^\mcm(s) = 1$. Moreover, by \eqref{iden1}, if $h >2$, $\phi_h^\mcm(s)$ has a Dirichlet series with abscissa of absolute convergence equal to $1/(2h+3)$. Thus, we may write
\begin{equation}\label{nhks}
    \mathfrak{N}_h(s) = \mathcal{G}_h(s) \mathcal{L}_h(s),
\end{equation}
where
\begin{equation}\label{L_h(s)}
\mathcal{L}_h(s) = \sum_{n=1}^\infty l_h(n) n^{-s} = \zeta_\mcm(hs) \zeta_\mcm((h+1)s) \cdots \zeta_\mcm((2h-1)s),
\end{equation}
and
\begin{equation}\label{G_h(s)}
\mathcal{G}_h(s) = \sum_{n=1}^\infty g_h(n) n^{-s} = \frac{\phi_h^\mcm(s)}{\zeta_\mcm((2h+2)s)},
\end{equation}
is a Dirichlet series converging absolutely in $\Re(s) > 1/(2h+2)$. Then, one can write
\begin{equation}\label{nhkx1}
    |\mathcal{N}_h(x)| = \sum_{\substack{\mathfrak{a}_0,\cdots, \mathfrak{a}_{h-1} \in \mathcal{M} \\ N(\mathfrak{a}_0)^h N(\mathfrak{a}_1)^{h+1} \cdots N(\mathfrak{a}_{h-1})^{2h -1} \leq x}} \mu_\mcm^2(\mathfrak{a}_1 + \cdots + \mathfrak{a}_{h-1}) = \sum_{m n \leq x} g_h(m) l_h(n).
\end{equation}
In the above expression, $m$ and $n$ would be natural numbers if $X = \mathbb{Q}$ and would be positive powers of $q$ if $X = \{q^z : z \in \mathbb{Z} \}$. 

Additionally, let $\mathcal{T}_h(x)$ denote the unweighted sum
\begin{equation}\label{T_h(x)}
\mathcal{T}_h(x) := \sum_{n \leq x} l_h(n) = \sum_{N(\mathfrak{a}_0)^h N(\mathfrak{a}_1)^{h+1} \cdots N(\mathfrak{a}_{h-1})^{2h -1} \leq x} 1,
\end{equation}
where the sum is taken over all ideals $\mathfrak{a}_0, \ldots, \mathfrak{a}_{h-1}$. Clearly, $|\mathcal{N}_h(x)| \leq \mathcal{T}_h(x)$. 
\subsection{Distributions of \texorpdfstring{$h$}{}-full elements}
To prove estimates involving $\mathcal{N}_h(x)$, we need the following estimate for $\mathcal{T}_h(x)$:
\begin{lma}\label{T_h(X)bound}\cite[Lemma 6.1]{dkl3}
For any $x \in X$ with $x > 1$ and any integer $h \geq 2$, we have
$$\mathcal{T}_h(x) = \kappa \left( \prod_{i=1}^{h-1} \zeta_\mcm \left( 1+i/h \right) \right) x^{1/h} + O_h \big( R_{\mathcal{T}_h}(x) \big),$$
where
\begin{equation}\label{E1(x)}
    R_{\mathcal{T}_h}(x) = \begin{cases}
        x^{\frac{\theta}{h}} & \text{ if } \frac{h}{h+1} < \theta, \\
        x^{\frac{1}{h+1}} (\log x) & \text{ if } \frac{h}{h+i} = \theta  \text{ for some } i \in \{ 1, \ldots, h-1 \}\\
        x^{1/(h+1)} & \text{ if } \frac{h}{h+1} > \theta \text{ and } \frac{h}{h+i} \neq \theta \text{ for any } i \in \{ 1, \cdots, h-1 \}.
    \end{cases}
    \end{equation}
\end{lma}
\begin{proof}
Using Condition ($\star$), we have
\begin{align}\label{Teq1}
    \mathcal{T}_h(x) & = \sum_{N(\mathfrak{a}_1)^{h+1} \cdots N(\mathfrak{a}_{h-1})^{2h -1} \leq x} \ \sum_{N(\mathfrak{a}_0)^h \leq \frac{x}{N(\mathfrak{a}_1)^{h+1} \cdots N(\mathfrak{a}_{h-1})^{2h -1}}} 1 \notag\\
    & = \sum_{N(\mathfrak{a}_1)^{h+1} \cdots N(\mathfrak{a}_{h-1})^{2h -1} \leq x}  \Bigg( \frac{\kappa}{(N(\mathfrak{a}_1)^{h+1} \cdots N(\mathfrak{a}_{h-1})^{2h -1})^{1/h}} x^{1/h} \notag \\
    & \hspace{5cm} + O \left( \left( \frac{x^{1/h}}{(N(\mathfrak{a}_1)^{h+1} \cdots N(\mathfrak{a}_{h-1})^{2h -1})^{1/h}} \right)^\theta \right) \Bigg) \notag\\
    & = \kappa x^{1/h} \mathfrak{T}_1 + O(\mathfrak{T}_2),
\end{align}    
where $\mathfrak{T}_1$ and $\mathfrak{T}_2$ are defined as
$$\mathfrak{T}_1 = \sum_{N(\mathfrak{a}_1)^{h+1} \cdots N(\mathfrak{a}_{h-1})^{2h -1} \leq x}  \frac{1}{(N(\mathfrak{a}_1)^{h+1} \cdots N(\mathfrak{a}_{h-1})^{2h -1})^{1/h}},$$
and
$$\mathfrak{T}_2 = x^{\theta/h} \sum_{N(\mathfrak{a}_1)^{h+1} \cdots N(\mathfrak{a}_{h-1})^{2h -1} \leq x} \frac{1}{(N(\mathfrak{a}_1)^{h+1} \cdots N(\mathfrak{a}_{h-1})^{2h -1})^{\theta/h}}.$$
Notice that, by Part 7 of \lmaref{boundnm}, we have
\begin{align*}
    \mathfrak{T}_1 & = \sum_{N(\mathfrak{a}_2)^{h+2} \cdots N(\mathfrak{a}_{h-1})^{2h -1} \leq x} \frac{1}{(N(\mathfrak{a}_2)^{h+2} \cdots N(\mathfrak{a}_{h-1})^{2h -1})^{1/h}} \\
    & \hspace{5cm} \times  \sum_{N ( \mathfrak{a}_1)^{h+1} \leq \frac{x}{N (\mathfrak{a}_2)^{h+2} \cdots N(\mathfrak{a}_{h-1})^{2h -1}}} \frac{1}{N(\mathfrak{a}_1)^{1+1/h}} \\
    & = \sum_{N(\mathfrak{a}_2)^{h+2} \cdots N(\mathfrak{a}_{h-1})^{2h -1} \leq x} \frac{1}{(N(\mathfrak{a}_2)^{h+2} \cdots N(\mathfrak{a}_{h-1})^{2h -1})^{1/h}} \bigg( \zeta_\mcm(1+1/h) \\
    & \hspace{4cm} + O \bigg( \sum_{N(\mathfrak{a}_1) > \frac{x^{1/(h+1)}}{\left(N (\mathfrak{a}_2)^{h+2} \cdots N(\mathfrak{a}_{h-1})^{2h -1} \right)^{1/(h+1)}}} \frac{1}{N(\mathfrak{a}_1)^{1+1/h}} \bigg) \bigg) \\
    & = \zeta_\mcm(1+1/h) \sum_{N(\mathfrak{a}_2)^{h+2} \cdots N(\mathfrak{a}_{h-1})^{2h -1} \leq x} \frac{1}{(N(\mathfrak{a}_2)^{h+2} \cdots N(\mathfrak{a}_{h-1})^{2h -1})^{1/h}} \\
    & \hspace{1cm} + O \left( \frac{1}{x^{1/(h(h+1))}}  \sum_{N(\mathfrak{a}_2)^{h+2} \cdots N(\mathfrak{a}_{h-1})^{2h -1} \leq x} \frac{1}{(N(\mathfrak{a}_2)^{h+2} \cdots N(\mathfrak{a}_{h-1})^{2h -1})^{1/(h+1)}} \right).
\end{align*}
Note that, for any $h \geq 2$, the sum inside the error term above is $O_h(1)$. Thus, we deduce
$$\mathfrak{T}_1 =  \zeta_\mcm(1+1/h) \sum_{N(\mathfrak{a}_2)^{h+2} \cdots N(\mathfrak{a}_{h-1})^{2h -1} \leq x} \frac{1}{(N(\mathfrak{a}_2)^{h+2} \cdots N(\mathfrak{a}_{h-1})^{2h -1})^{1/h}} +  O_h \left( \frac{1}{x^{1/(h(h+1))}} \right).$$
Repeating this process $h-1$ times and realizing that $x^{-i/(h(h+i))} \leq x^{-1/(h(h+1))}$ for all $i \geq 1$, we obtain
\begin{align*}
    \mathfrak{T}_1 & = \prod_{i=1}^{2} \zeta_\mcm(1+i/h) \sum_{N(\mathfrak{a}_3^{h+3} \cdots \mathfrak{a}_{h-1}^{2h -1}) \leq x} \frac{1}{(N(\mathfrak{a}_3^{h+3} \cdots \mathfrak{a}_{h-1}^{2h -1}))^{1/h}} +  O_h \left( \frac{1}{x^{1/(h(h+1))}} \right) \notag \\
    & \vdots \notag \\
    & = \prod_{i=1}^{h-2} \zeta_\mcm(1+i/h) \sum_{N(\mathfrak{a}_{h-1})^{2h -1} \leq x} \frac{1}{(N(\mathfrak{a}_{h-1})^{2h -1})^{1/h}} +  O_h \left( \frac{1}{x^{1/(h(h+1))}} \right) \notag \\
    & = \prod_{i=1}^{h-1} \zeta_\mcm(1+i/h) + O_h \left( \frac{1}{x^{1/(h(h+1))}} \right).
\end{align*}
Thus,
\begin{equation}\label{bound-T1}
\kappa x^{1/h} \mathfrak{T}_1 = \kappa \left( \prod_{i=1}^{h-1} \zeta_\mcm \left( 1+i/h \right) \right) x^{1/h} + O_h \left( x^{1/(h+1)} \right).   
\end{equation}
Next, we bound the term $\mathfrak{T}_2$. 
First, we assume $\theta(1+i/h) > 1$. Then $\theta(1+i/h) > 1$ for all $i \in \{1, \dots, h-1 \}$. Thus, by Part (3) of \lmaref{boundnm}, we have
\begin{align*}
    \mathfrak{T}_2 & = x^{\theta/h} \sum_{N(\mathfrak{a}_1^{h+1} \cdots \mathfrak{a}_{h-1}^{2h -1}) \leq x} \frac{1}{(N(\mathfrak{a}_1^{h+1} \cdots \mathfrak{a}_{h-1}^{2h -1}))^{\theta/h}} \\
    & = x^{\theta/h} \sum_{N(\mathfrak{a}_1^{h+1} \cdots \mathfrak{a}_{h-2}^{2h -2}) \leq x} \frac{1}{(N(\mathfrak{a}_1^{h+1} \cdots \mathfrak{a}_{h-2}^{2h -2}))^{\theta/h}} \sum_{N(\mathfrak{a}_{h-1}) \leq \frac{x^{1/(2h-1)}}{N(\mathfrak{a}_1^{h+1} \cdots \mathfrak{a}_{h-2}^{2h -2})^{1/(2h-1)}}} \frac{1}{\mathfrak{a}_{h-1}^{(2h -1)\theta/h}} \\
    & \ll_h x^{\theta/h} \sum_{N(\mathfrak{a}_1^{h+1} \cdots \mathfrak{a}_{h-2}^{2h -2}) \leq x} \frac{1}{(N(\mathfrak{a}_1^{h+1} \cdots \mathfrak{a}_{h-2}^{2h -2}))^{\theta/h}},
\end{align*}
and repeating the step $h-2$ times, we obtain
\begin{equation}\label{1}
    \mathfrak{T}_2 \ll_h x^{\theta/h} \sum_{N(\mathfrak{a}_1^{h+1}) \leq x} \frac{1}{(N(\mathfrak{a}_1^{h+1})^{\theta/h}} \ll_h x^{\theta/h}.
\end{equation}
Next, we assume $\theta(1+i/h) = 1$ for some $i \in \{1, \cdots, h-1 \}$. Thus, $\theta(1+k/h) > 1$ for $k > i$ and $\theta(1+k/h) < 1$ for $k < i$. Then, working similarly to the above parts and using Parts (2), (3), and (5) of \lmaref{boundnm} above, we have
\begin{align*}
    \mathfrak{T}_2 & = x^{\theta/h} \sum_{N(\mathfrak{a}_1^{h+1} \cdots \mathfrak{a}_{h-1}^{2h -1}) \leq x} \frac{1}{(N(\mathfrak{a}_1^{h+1} \cdots \mathfrak{a}_{h-1}^{2h -1}))^{\theta/h}} \\
    & \ll_h x^{\theta/h} \sum_{N(\mathfrak{a}_1^{h+1} \cdots \mathfrak{a}_{i}^{h+i}) \leq x} \frac{1}{(N(\mathfrak{a}_1^{h+1} \cdots \mathfrak{a}_{i}^{h+i}))^{\theta/h}}  \\
    & \ll_h x^{\theta/h} \sum_{N(\mathfrak{a}_1^{h+1} \cdots \mathfrak{a}_{i-1}^{h+i-1 }) \leq x} \frac{1}{(N(\mathfrak{a}_1^{h+1} \cdots \mathfrak{a}_{i-1}^{h+i-1}))^{\theta/h}} \sum_{N(\mathfrak{a}_{i}) \leq \frac{x^{1/(h+i)}}{N(\mathfrak{a}_1^{h+1} \cdots \mathfrak{a}_{i-1}^{h+i-1 })^{1/(h+i)}}} \frac{1}{(N(\mathfrak{a}_{i}^{h+i}))^{\theta/h}} \\
    & \ll_h x^{\theta/h} (\log x) \sum_{N(\mathfrak{a}_1^{h+1} \cdots \mathfrak{a}_{i-1}^{h+i-1 }) \leq x} \frac{1}{(N(\mathfrak{a}_1^{h+1} \cdots \mathfrak{a}_{i-1}^{h+i-1}))^{\theta/h}} \\
    & \ll_h x^{1/(h+i-1)} (\log x) \sum_{N(\mathfrak{a}_1^{h+1} \cdots \mathfrak{a}_{i-2}^{h+i-2}) \leq x} \frac{1}{N(\mathfrak{a}_1^{h+1} \cdots \mathfrak{a}_{i-2}^{h+i-2 })^{1/(h+i-1)}}.
\end{align*}
Repeating the process, we obtain
\begin{equation}\label{2}
    \mathfrak{T}_2 \ll_h x^{1/(h+2)} (\log x) \sum_{N(\mathfrak{a}_1) \leq x^{1/(h+1)}} \frac{1}{N(\mathfrak{a}_1^{h+1})^{1/(h+2)}} \ll_h x^{1/(h+1)} (\log x).
\end{equation}
Finally, for the third case, we assume $\theta(1+1/h) < 1 \text{ and } \theta(1+i/h) \neq 1 \text{ for any } i \in \{ 1, \cdots, h-1 \}$. Thus, there exists a $k \in \{1, \cdots, h-1 \}$ such that $\upsilon(1+r/h) < 1$ for $r \leq k$ and $\theta(1+r/h) > 1$ for $k < r < h-1$. Thus, using Parts (2) and (3) of \lmaref{boundnm} and separating the sums similarly to the previous two cases, we obtain
\begin{align*}
    \mathfrak{T}_2 & \ll_h x^{\theta/h} \sum_{N(\mathfrak{a}_1^{h+1} \cdots \mathfrak{a}_{k}^{h+k}) \leq x} \frac{1}{(N(\mathfrak{a}_1^{h+1} \cdots \mathfrak{a}_{k}^{h+k}))^{\theta/h}}  \\
    & \ll_h x^{1/(h+k)} \sum_{N(\mathfrak{a}_1^{h+1} \cdots \mathfrak{a}_{k-1}^{h+k-1}) \leq x} \frac{1}{N(\mathfrak{a}_1^{h+1} \cdots \mathfrak{a}_{k-1}^{h+k-1 })^{1/(h+k)}}.
\end{align*}
Repeating the steps again, we obtain
\begin{equation}\label{3}
    \mathfrak{T}_2 \ll_h x^{1/(h+2)} \sum_{N(\mathfrak{a}_1) \leq x^{1/(h+1)}} \frac{1}{N(\mathfrak{a}_1^{h+1})^{1/(h+2)}} \ll_h x^{1/(h+1)}.
\end{equation}
Combining \eqref{Teq1}, \eqref{bound-T1}, \eqref{1}, \eqref{2}, and \eqref{3} completes the proof.
\end{proof}

We prove the following distributions of \texorpdfstring{$h$}{}-full elements:
\begin{proof}[\textbf{Proof of \thmref{hfullideals}}]
Using \eqref{nhkx1}, \eqref{T_h(x)} and \lmaref{T_h(X)bound}, we have
\begin{align}\label{hfulleqn1}
    |\mathcal{N}_h(x)| & = \sum_{m \leq x} g_h(m) \sum_{n \leq x/m} l_h(n) \notag \\
    & = \sum_{m \leq x} g_h(m) \mathcal{T}_h(x/m) \notag \\
    & = \kappa \left( \prod_{i=1}^{h-1} \zeta_\mcm \left( 1+i/h \right) \right) x^{1/h} \sum_{m \leq x} \frac{g_h(m)}{m^{1/h}} + O_h \left( \sum_{m \leq x} g_h(m) R_{\mathcal{T}_h}(x/m) \right).
\end{align}
Recall that 
$$\mathcal{G}_h(s) = \sum_{n=1}^\infty g_h(n) n^{-s} = \frac{\phi_h^\mcm(s)}{\zeta_\mcm((2h+2)s)}$$
converges absolutely in $\Re(s) > 1/(2h+2)$. Therefore, for any $\epsilon > 0$,
$$\sum_{n \leq x} g_h(n) \ll x^{(1/(2h+2))+ \epsilon},$$
and thus, by partial summation,
\begin{equation}\label{hfulleqn2}
    \sum_{m \leq x} \frac{g_h(m)}{m^{1/h}} = \mathcal{G}_h (1/h) + O \left( \sum_{m > x} \frac{g_h(m)}{m^{1/h}} \right) = \mathcal{G}_h(1/h) + O(x^{(1/(2h+2)) + \epsilon - 1/h}).
\end{equation}
Moreover, by the definition of $R_{\mathcal{T}_h}(x)$ \eqref{E1(x)}, we have
\begin{scriptsize}
\begin{equation*}
    \sum_{m \leq x} g_h(m) R_{\mathcal{T}_h}(x/m) = \begin{cases}
        x^{\frac{\theta}{h}} \sum_{m \leq x} \frac{g_h(m)}{m^{\theta/h}} & \text{ if } \frac{h}{h+1} < \theta, \\
        x^{\frac{1}{h+1}} \sum_{m \leq x} \frac{g_h(m)}{m^{1/(h+1)}} (\log (x/m)) & \text{ if } \frac{h}{h+i} = \theta  \text{ for some } i \in \{ 1, \ldots, h-1 \}, \\
        x^{\frac{1}{h+1}} \sum_{m \leq x} \frac{g_h(m)}{m^{1/(h+1)}} & \text{ if } \frac{h}{h+1} > \theta \text{ and } \frac{h}{h+i} \neq \theta \text{ for any } i \in \{ 1, \cdots, h-1 \}.
    \end{cases}
    \end{equation*}
\end{scriptsize}
Again, since $\mathcal{G}_h(s)$ converges at $1/(h+1)$, for the third case above, we have
$$x^{\frac{1}{h+1}} \sum_{m \leq x} \frac{g_h(m)}{m^{1/(h+1)}} \ll_h x^{\frac{1}{h+1}},$$
and for the middle case, we have
\begin{equation*}
    x^{\frac{1}{h+1}}\sum_{m \leq x} \frac{g_h(m)}{m^{1/(h+1)}} (\log (x/m)) \ll_h x^{\frac{1}{h+1}} (\log x).
\end{equation*}
Moreover, for the first case, since $\theta/h > 1/(h+1)$, thus
$$x^{\frac{\theta}{h}} \sum_{m \leq x} \frac{g_h(m)}{m^{\theta/h}} \ll x^{\frac{\theta}{h}} \sum_{m \leq x} \frac{g_h(m)}{m^{1/(h+1)}} \ll_h x^{\frac{\theta}{h}}.$$
Putting the last four results together, we have
\begin{small}
\begin{equation}\label{sumwithE_1}
    \sum_{m \leq x} g_h(m) R_{\mathcal{T}_h}(x/m) \ll_h \begin{cases}
        x^{\frac{\theta}{h}} & \text{ if } \frac{h}{h+1} < \theta, \\
        x^{\frac{1}{h+1}} (\log x) & \text{ if } \frac{h}{h+i} = \theta  \text{ for some } i \in \{ 1, \ldots, h-1 \}, \\
        x^{\frac{1}{h+1}} & \text{ if } \frac{h}{h+1} > \theta \text{ and } \frac{h}{h+i} \neq \theta \text{ for any } i \in \{ 1, \cdots, h-1 \}.
    \end{cases} 
\end{equation}
\end{small}
Combining the above together with \eqref{hfulleqn1} and \eqref{hfulleqn2}, we obtain
\begin{equation*}
    |\mathcal{N}_h(x)| = \kappa \left( \prod_{i=1}^{h-1} \zeta_\mcm \left( 1+i/h \right) \right) \mathcal{G}_h(1/h) x^{1/h} + O \big( R_{\mathcal{N}_h}(x) \big),
\end{equation*}
where $R_{\mathcal{N}_h}(x)$ is defined in \eqref{E2(x)}. Finally, by \eqref{Nhxformula}, \eqref{formulanhk} and \eqref{dedekind}, we have
\begin{align*}
\left( \prod_{i=1}^{h-1} \zeta_\mcm \left( 1+i/h \right) \right) \mathcal{G}_h(1/h) & = \prod_\mfp \left( 1 +  \frac{N(\mfp)^{-1}}{1 - N(\mfp)^{-1/h}} \right) \left( 1 - \frac{1}{N(\mfp)} \right) \\
& = \prod_\mfp \left( 1 + \frac{N(\mfp) - N(\mfp)^{1/h}}{N(\mfp)^2 \left( N(\mfp)^{1/h} - 1 \right)}\right),
\end{align*}
which is a convergent product, since the left-hand side has a finite value. This completes the proof.
\end{proof}
\begin{lma}\label{hfullidealsrestrict}
Let $x \in X$ and let $h \geq 2$ be an integer. Let $\ell_1, \cdots, \ell_r$ be fixed prime elements and $\mathcal{N}_{h,\ell_1,\cdots, \ell_r}(x)$ denote the set of $h$-full elements with norm $N(\cdot)$ less than or equal to $x$ and co-prime to $\ell_i$ for all $i \in \{1,\cdots, r\}$. Then, we have
$$|\mathcal{N}_{h,\ell_1,\cdots, \ell_r}(x)| = \prod_{i=1}^r \frac{\kappa \gamma_{\scaleto{h}{4.5pt}}}{\left( 1+ \frac{N(\ell_i)^{-1}}{1-N(\ell_i)^{-1/h}} \right)} x^{1/h} + O_{h,r} \big( R_{\mathcal{N}_h}(x) \big),$$
where $\gamma_{\scaleto{h}{4.5pt}}$ is defined in \eqref{gammahk},
and where $R_{\mathcal{N}_h}(x)$ is defined in \eqref{E2(x)}.
\end{lma}
\begin{proof}
Similar to \eqref{formulanhk}, we have
\begin{equation*}
   \mathfrak{N}_{h,\ell_1,\cdots,\ell_r}(s) := \sum_{\substack{\mathfrak{m} \in \mathcal{N}_h \\ n_{\ell_i}(\mfm) = 0, \ \forall i }} \frac{1}{N(\mathfrak{m})^s} = \prod_{\substack{\mfp \\ \mfp \neq \ell_i, \ \forall i }} \left( 1 +  \frac{N(\mfp)^{-hs}}{1 - N(\mfp)^{-s}} \right). 
\end{equation*}
Moreover, by \eqref{nhks}, we can write
$$\mathfrak{N}_{h,\ell_1,\cdots,\ell_r}(s) = \mathcal{G}_{h,\ell_1,\cdots,\ell_r}(s) \mathcal{L}_h(s),$$
where $\mathcal{L}_h(s)$ is defined in \eqref{L_h(s)} and
\begin{equation*}
\mathcal{G}_{h,\ell_1,\cdots,\ell_r}(s) = \sum_{n \leq x} g_{h,\ell_1,\cdots,\ell_r}(n) n^{-s}= \frac{\phi_h^\mcm(s)}{ \prod_{i=1}^r \left( 1 +  \frac{N(\ell_i)^{-hs}}{1 - N(\ell_i)^{-s}} \right)\zeta_\mcm((2h+2)s)},
\end{equation*}
where $\phi_h^\mcm(s)$ satisfies \eqref{iden1}. We complete the proof by working similarly to the proof of \thmref{hfullideals} and using
\begin{align*}
\left( \prod_{i=1}^{h-1} \zeta_\mcm \left( 1+i/h \right) \right) \mathcal{G}_{h,\ell_1,\cdots,\ell_r}(1/h) & = \prod_{\substack{\mfp \\ \mfp \neq \ell_i \ \forall i}}  \left( 1 +  \frac{N(\mfp)^{-1}}{1 - N(\mfp)^{-1/h}} \right) \prod_\mfp \left( 1 - \frac{1}{N(\mfp)} \right) \\
& = \frac{\gamma_{\scaleto{h}{4.5pt}}}{\prod_{i=1}^r \left( 1+ \frac{N(\ell_i)^{-1}}{1-N(\ell_i)^{-1/h}} \right)}.
\end{align*}
\end{proof}
\subsection{Distribution of \texorpdfstring{$\omega(\mfm)$}{} over \texorpdfstring{$h$}{}-full elements}
For the distributions of $\omega(\mfm)$ over $h$-full elements, we prove:
\begin{proof}[\textbf{Proof of \thmref{hfullomega}}]
Inserting the formula for $\mathcal{N}_{h,\mfp}(y)$ given in \lmaref{hfullidealsrestrict} with $y = x/N(\mfp)^k$ and with a prime element $\mfp$, we obtain
\begin{align}\label{mainomegahfull1}
& \sum_{\mfm \in \mathcal{N}_h(x)} \omega(\mfm) \notag\\
& = \sum_{N(\mfp) \leq x^{1/h}} \sum_{k = h}^{\left\lfloor \frac{\log x}{\log N(\mfp)}  \right\rfloor} | \mathcal{N}_{h,\mfp}(x/N(\mfp)^k)| \notag \\
& = \sum_{N(\mfp) \leq x^{1/h}} \sum_{k = h}^{\left\lfloor \frac{\log x}{\log N(\mfp)}  \right\rfloor} \left( \frac{\kappa \gamma_{\scaleto{h}{4.5pt}}}{\left( 1+ \frac{N(\mfp)^{-1}}{1-N(\mfp)^{-1/h}} \right)} \frac{x^{1/h}}{N(\mfp)^{k/h}} + O_{h} \big( R_{\mathcal{N}_h}(x/N(\mfp)^k) \big) \right).
\end{align}
Let us study the main term above given by
$$\kappa \gamma_{\scaleto{h}{4.5pt}} x^{1/h} \sum_{N(\mfp) \leq x^{1/h}} \sum_{k = h}^{\left\lfloor \frac{\log x}{\log N(\mfp)}  \right\rfloor} \frac{1}{N(\mfp)^{k/h} \left( 1+ \frac{N(\mfp)^{-1}}{1-N(\mfp)^{-1/h}} \right)}.$$
We obtain
\begin{align}\label{sum_need}
& \kappa \gamma_{\scaleto{h}{4.5pt}} x^{1/h} \sum_{N(\mfp) \leq x^{1/h}} \sum_{k = h}^{\left\lfloor \frac{\log x}{\log N(\mfp)}  \right\rfloor} \frac{1}{N(\mfp)^{k/h} \left( 1+ \frac{N(\mfp)^{-1}}{1-N(\mfp)^{-1/h}} \right)} \notag \\
& = \kappa \gamma_{h} x^{1/h} \sum_{N(\mfp) \leq x^{1/h}} \frac{1}{N(\mfp) \left( 1- N(\mfp)^{-1/h}+N(\mfp)^{-1} \right)} \notag \\
& \hspace{3cm} -  \kappa \gamma_{h} x^{1/h} \sum_{N(\mfp) \leq x^{1/h}} \frac{(N(\mfp)^{-1/h})^{\left\lfloor \frac{\log x}{\log N(\mfp)}  \right\rfloor - h +1}}{N(\mfp) \left( 1- N(\mfp)^{-1/h}+  N(\mfp)^{-1} \right)}.
\end{align}
Using $\lfloor x \rfloor \geq x-1$, and $\eqref{gpnt}$, we bound the second term above with
\begin{equation}\label{needme}
\kappa \gamma_{h} x^{1/h} \sum_{N(\mfp) \leq x^{1/h}} \frac{(N(\mfp)^{-1/h})^{\left\lfloor \frac{\log x}{\log N(\mfp)}  \right\rfloor - h +1}}{N(\mfp) \left( 1- N(\mfp)^{-1/h}+  N(\mfp)^{-1} \right)} \ll_h \frac{x^{1/h}}{\log x}.
\end{equation}
Thus, it remains to study the first term in \eqref{sum_need}. A little manipulation yields
\begin{align}\label{req3}
& \sum_{N(\mfp) \leq x^{1/h}} \frac{1}{N(\mfp) \left( 1- N(\mfp)^{-1/h}+N(\mfp)^{-1} \right)} \notag \\
& = \sum_{N(\mfp) \leq x^{1/h}} \frac{1}{N(\mfp)} + \sum_{N(\mfp) \leq x^{1/h}} \frac{N(\mfp)^{-1/h} - N(\mfp)^{-1}}{ N(\mfp) \left( 1- N(\mfp)^{-1/h} + N(\mfp)^{-1} \right) }.
\end{align}
For the first sum on the right-hand side above, we use Part 4 of \lmaref{boundnm} to obtain
\begin{equation}\label{req2}
\sum_{N(\mfp) \leq x^{1/h}} \frac{1}{N(\mfp)} = \log \log x - \log h + \mathfrak{A} + O_h \left( \frac{1}{\log x} \right).
\end{equation}
For the second sum on the right-hand side of \eqref{req3}, we use the convergence of the corresponding infinite series and Part 6 of \lmaref{boundnm} to obtain
\begin{align*}
\sum_{N(\mfp) \leq x^{1/h}}  \frac{N(\mfp)^{-1/h} - N(\mfp)^{-1}}{ N(\mfp) \left( 1- N(\mfp)^{-1/h} + N(\mfp)^{-1} \right) } 
& = \mathfrak{L}_h(h+1) - \mathfrak{L}_h(2h) + O_h \left( \frac{1}{x^{1/h^2} (\log x)} \right).
\end{align*}  

Combining the last three results, we obtain
\begin{align}\label{req_bound_p}
    \sum_{N(\mfp) \leq x^{1/h}} \frac{1}{N(\mfp) \left( 1- N(\mfp)^{-1/h}+N(\mfp)^{-1} \right)} & = \log \log x + \mathfrak{D}_1
    + O_h \left( \frac{1}{\log x} \right),
\end{align}
where $\mathfrak{D}_1$ is defined in \eqref{D1}. Combining the above result with \eqref{mainomegahfull1}, \eqref{sum_need}, and \eqref{needme}, we obtain
\begin{align}\label{hfullfinal}
\sum_{\mfm \in \mathcal{N}_h(x)} \omega(\mfm) & =  \kappa \gamma_{h} x^{1/h} \log \log x + \kappa \mathfrak{D}_1 \gamma_{h} x^{1/h}  + O_h \left( \frac{x^{1/h}}{\log x} \right)  \notag \\
& \hspace{.5cm} + O_h \left( \sum_{N(\mfp) \leq x^{1/h}} \sum_{k = h}^{\left\lfloor \frac{\log x}{\log N(\mfp)}  \right\rfloor} \left( R_{\mathcal{N}_h}(x/N(\mfp)^k) \right)\right).
\end{align}
By \eqref{E2(x)}, we can write $R_{\mathcal{N}_h}(x) \ll x^{\upsilon/h}$ where $0 < \upsilon < 1$. Using this in the error term above, and Part 1 of \lmaref{boundnm} with $\alpha = \upsilon$, we obtain
\begin{align*}
    \sum_{N(\mfp) \leq x^{1/h}} \sum_{k = h}^{\left\lfloor \frac{\log x}{\log N(\mfp)}  \right\rfloor} \left( R_{\mathcal{N}_h}(x/N(\mfp)^k) \right) & \ll x^{\upsilon/h} \sum_{N(\mfp) \leq x^{1/h}} \frac{1}{N(\mfp)^\upsilon} \ll_h \frac{x^{1/h}}{\log x}.
\end{align*}
Inserting the above back into \eqref{hfullfinal} completes the first part of the proof.

Now, note that
\begin{align}\label{hfullp1}
    \sum_{\mfm \in \mathcal{N}_h(x)} \omega^2(\mfm)
    & = \sum_{\mfm \in \mathcal{N}_h(x)} \left( \sum_{\substack{\mfp \\ n_\mfp(\mfm) \geq h}} 1 \right)^2 \notag \\
    & = \sum_{\mfm \in \mathcal{N}_h(x)} \omega(\mfm) + \sum_{\mfm \in \mathcal{N}_h(x)}  \sum_{\substack{\mfp,\mfq \\ \mfp^k || \mfm, \ \mfq^l || \mfm, \ \mfp \neq \mfq }} \left( \sum_{k=h}^{\left\lfloor \frac{\log x}{\log N(\mfp)}  \right\rfloor} \sum_{l=h}^{\left\lfloor \frac{\log x}{\log N(\mfq)}  \right\rfloor} 1 \right).
\end{align}
The first sum on the right side above is the first moment studied earlier. For the second sum, we first rewrite the sum, and then use \lmaref{hfullidealsrestrict} with two distinct prime ideals $\mfp$ and $\mfq$ to obtain
\begin{align}\label{hfullp2}
    & \sum_{\mfm \in \mathcal{N}_h(x)} \sum_{\substack{\mfp,\mfq \\ \mfp^k || \mfm, \ \mfq^l || \mfm, \ \mfp \neq \mfq }} \left( \sum_{k=h}^{\left\lfloor \frac{\log x}{\log N(\mfp)}  \right\rfloor} \sum_{l=h}^{\left\lfloor \frac{\log x}{\log N(\mfq)}  \right\rfloor} 1 \right) \notag \\
    & = \kappa \gamma_{h} x^{1/h} \sum_{\substack{\mfp, \mfq \\ \mfp \neq \mfq, \ N(\mfp) N(\mfq) \leq x^{1/h}}} \Bigg( \sum_{k=h}^{\left\lfloor \frac{\log x}{\log N(\mfp)}  \right\rfloor} \sum_{l=h}^{\left\lfloor \frac{\log x}{\log N(\mfq)}  \right\rfloor} \frac{1}{N(\mfp)^{k/h} \left( 1 + \frac{N(\mfp)^{-1}}{1 - N(\mfp)^{-1/h}} \right)} \notag \\
    & \hspace{1cm} \cdot \frac{1}{N(\mfq)^{l/h}  \left( 1 + \frac{N(\mfq)^{-1}}{1 - N(\mfq)^{-1/h}} \right)} \Bigg) \notag \\
    & \hspace{.5cm} + O_h \left( \sum_{\substack{\mfp, \mfq \\ \mfp \neq \mfq, \ N(\mfp) N(\mfq) \leq x^{1/h}}} \Bigg( \sum_{k=h}^{\left\lfloor \frac{\log x}{\log N(\mfp)}  \right\rfloor} \sum_{l=h}^{\left\lfloor \frac{\log x}{\log N(\mfq)}  \right\rfloor} \big( R_{\mathcal{N}_h}(x/(N(\mfp)^k N(\mfq)^l)  ) \big) \Bigg) \right).
\end{align}
Recall that we can write $R_{\mathcal{N}_h}(x) \ll x^{\upsilon/h}$ where $0 < \upsilon < 1$. Thus, for $X = \mathbb{Q}$, the error term above is bounded by using Part 1 of \lmaref{boundnm} and \lmaref{sumplogp} as the following
\begin{align}\label{hfullp3}
    & \sum_{\substack{\mfp, \mfq \\ \mfp \neq \mfq, \ N(\mfp) N(\mfq) \leq x^{1/h}}} \Bigg( \sum_{k=h}^{\left\lfloor \frac{\log x}{\log N(\mfp)}  \right\rfloor} \sum_{l=h}^{\left\lfloor \frac{\log x}{\log N(\mfq)}  \right\rfloor} \big( R_{\mathcal{N}_h}(x/(N(\mfp)^k N(\mfq)^l)  ) \big) \Bigg) \notag \\
     & \ll_h x^{\frac{\upsilon}{h}} \sum_{\substack{\mfp \\ N(\mfp) \leq x^{1/h}/2}} \frac{1}{N(\mfp)^{\upsilon}}  \sum_{\substack{\mfq \\ N(\mfq) \leq x^{1/h}/N(\mfp)}} \frac{1}{N(\mfq)^{\upsilon}}  \notag \\
     & \ll_h x^{\frac{1}{h}} \sum_{\substack{\mfp \\ N(\mfp) \leq x^{1/h}/2}} \frac{1}{N(\mfp) \log(x^{1/h}/N(\mfp))} \notag\\
     & \ll_h \frac{x^{\frac{1}{h}} \log \log x}{\log x}.
\end{align}
For $X = \{ q^z : z \in \mathbb{Z} \}$, a similar result as above can be proved using $x/q$ instead of $x/2$, and using the corresponding result from \lmaref{sumplogp}.

Next, we estimate the main term in \eqref{hfullp2}. First note that
\begin{align*}
    & \sum_{k=h}^{\left\lfloor \frac{\log x}{\log N(\mfp)}  \right\rfloor} \frac{1}{N(\mfp)^{k/h} \left( 1 + \frac{N(\mfp)^{-1}}{1 - N(\mfp)^{-1/h}} \right)} \\
    & = \frac{1}{N(\mfp)(1-N(\mfp)^{-1/h}+N(\mfp)^{-1})} - \frac{N(\mfp)^{-\frac{1}{h} \left( \left\lfloor \frac{\log x}{\log N(\mfp)} \right\rfloor - h + 1\right)}}{N(\mfp)(1-N(\mfp)^{-1/h}+N(\mfp)^{-1})}.
\end{align*}
Thus, using a similar result for a prime ideal $\mfq$ as above and the symmetry of two prime ideals $\mfp$ and $\mfq$, we deduce
\begin{align*}
    & \kappa \gamma_{h} x^{1/h} \sum_{\substack{\mfp, \mfq \\ \mfp \neq \mfq, \ N(\mfp) N(\mfq) \leq x^{1/h}}} \Bigg( \sum_{k=h}^{\left\lfloor \frac{\log x}{\log N(\mfp)}  \right\rfloor} \sum_{l=h}^{\left\lfloor \frac{\log x}{\log N(\mfq)}  \right\rfloor} \frac{1}{N(\mfp)^{k/h} \left( 1 + \frac{N(\mfp)^{-1}}{1 - N(\mfp)^{-1/h}} \right)} \notag \\
    & \hspace{1cm} \cdot \frac{1}{N(\mfq)^{l/h}  \left( 1 + \frac{N(\mfq)^{-1}}{1 - N(\mfq)^{-1/h}} \right)} \Bigg) \notag \\
    & =  \kappa \gamma_{h} x^{1/h} \sum_{\substack{\mfp, \mfq \\ \mfp \neq \mfq, \ N(\mfp) N(\mfq) \leq x^{1/h}}} \frac{1}{N(\mfp)(1-N(\mfp)^{-1/h}+N(\mfp)^{-1})} \frac{1}{N(\mfq)(1-N(\mfq)^{-1/h}+N(\mfq)^{-1})} \notag \\
    & \hspace{.5cm} - 2 J_1 + J_2,
\end{align*}
where
$$J_1 = \kappa \gamma_{h} x^{1/h}  \sum_{\substack{\mfp, \mfq \\ \mfp \neq \mfq, \ N(\mfp) N(\mfq) \leq x^{1/h}}} \frac{1}{N(\mfp)(1-N(\mfp)^{-1/h}+N(\mfp)^{-1})}   \frac{N(\mfq)^{-\frac{1}{h} \left( \left\lfloor \frac{\log x}{\log N(\mfq)} \right\rfloor - h + 1\right)}}{N(\mfq)(1-N(\mfq)^{-1/h}+N(\mfq)^{-1})},$$
and
$$J_2 = \kappa \gamma_{h} x^{1/h}  \sum_{\substack{\mfp, \mfq \\ \mfp \neq \mfq, \ N(\mfp) N(\mfq) \leq x^{1/h}}} \frac{N(\mfp)^{-\frac{1}{h} \left( \left\lfloor \frac{\log x}{\log N(\mfp)} \right\rfloor - h + 1\right)}}{N(\mfp)(1-N(\mfp)^{-1/h}+N(\mfp)^{-1})} \frac{N(\mfp)^{-\frac{1}{h} \left( \left\lfloor \frac{\log x}{\log N(\mfq)} \right\rfloor - h + 1\right)}}{N(\mfq)(1-N(\mfq)^{-1/h}+N(\mfq)^{-1})}.$$
Using $\lfloor x \rfloor \geq x-1$, $\eqref{gpnt}$ and \lmaref{sumplogp}, we obtain
\begin{align*}
    J_1 & 
    \ll_h \frac{x^{1/h} \log \log x}{\log x},
\end{align*}
and
\begin{align*}
    J_2 & 
    \ll_h \frac{\log \log x}{\log x}.
\end{align*}
Combining the last three results, we obtain
\begin{align}\label{hfullp4}
    & \kappa \gamma_{h} x^{1/h} \sum_{\substack{\mfp, \mfq \\ \mfp \neq \mfq, \ N(\mfp) N(\mfq) \leq x^{1/h}}} \Bigg( \sum_{k=h}^{\left\lfloor \frac{\log x}{\log N(\mfp)}  \right\rfloor} \sum_{l=h}^{\left\lfloor \frac{\log x}{\log N(\mfq)}  \right\rfloor} \frac{1}{N(\mfp)^{k/h} \left( 1 + \frac{N(\mfp)^{-1}}{1 - N(\mfp)^{-1/h}} \right)} \notag \\
    & \hspace{1cm} \cdot \frac{1}{N(\mfq)^{l/h}  \left( 1 + \frac{N(\mfq)^{-1}}{1 - N(\mfq)^{-1/h}} \right)} \Bigg) \notag \\
    & =  \kappa \gamma_{h} x^{1/h} \sum_{\substack{\mfp, \mfq \\ \mfp \neq \mfq, \ N(\mfp) N(\mfq) \leq x^{1/h}}} \frac{1}{N(\mfp)(1-N(\mfp)^{-1/h}+N(\mfp)^{-1})} \frac{1}{N(\mfq)(1-N(\mfq)^{-1/h}+N(\mfq)^{-1})} \notag \\
    & \hspace{.5cm} + O_h \left( \frac{x^{1/h} \log \log x}{\log x} \right).
\end{align}
Thus to complete the proof, we only require to estimate the main term in \eqref{hfullp4}. Using Part 6 of \lmaref{boundnm}, we have
\begin{align}\label{hfullp5}
    & \sum_{\substack{\mfp, \mfq \\ \mfp \neq \mfq, \ N(\mfp) N(\mfq) \leq x^{1/h}}} \frac{1}{N(\mfp)(1-N(\mfp)^{-1/h}+N(\mfp)^{-1})} \frac{1}{N(\mfq)(1-N(\mfq)^{-1/h}+N(\mfq)^{-1})} \notag \\
    & = \sum_{\substack{\mfp, \mfq \\ N(\mfp) N(\mfq) \leq x^{1/h}}} \frac{1}{N(\mfp)(1-N(\mfp)^{-1/h}+N(\mfp)^{-1})} \frac{1}{N(\mfq)(1-N(\mfq)^{-1/h}+N(\mfq)^{-1})} \notag \\
    & \hspace{.5cm} - \sum_{\mfp} \left( \frac{1}{N(\mfp)-N(\mfp)^{1-1/h}+1} \right)^2 + O_h \left( \frac{1}{x^{1/(2h)} \log x} \right).
\end{align}
Now, using
$$\frac{1}{N(\mfp)(1-N(\mfp)^{-1/h}+N(\mfp)^{-1})} = \frac{1}{N(\mfp)} + \frac{N(\mfp)^{-1/h} - N(\mfp)^{-1}}{N(\mfp)(1-N(\mfp)^{-1/h}+N(\mfp)^{-1})},$$
a similar result for another prime ideal $\mfq$, and the symmetry of primes $\mfp$ and $\mfq$, we write the first sum on the right-hand side of \eqref{hfullp5} as
\begin{align*}
    & \sum_{\substack{\mfp, \mfq \\ N(\mfp) N(\mfq) \leq x^{1/h}}} \frac{1}{N(\mfp)(1-N(\mfp)^{-1/h}+N(\mfp)^{-1})} \frac{1}{N(\mfq)(1-N(\mfq)^{-1/h}+N(\mfq)^{-1})} \\
    & = \sum_{\substack{\mfp, \mfq \\ N(\mfp) N(\mfq) \leq x^{1/h}}} \frac{1}{N(\mfp) N(\mfq)} + 2 \sum_{\substack{\mfp, \mfq \\ N(\mfp) N(\mfq) \leq x^{1/h}}} \frac{N(\mfq)^{-1/h} - N(\mfq)^{-1}}{N(\mfp) N(\mfq)(1-N(\mfq)^{-1/h}+N(\mfq)^{-1})} \\
    & \hspace{.5cm} + \sum_{\substack{\mfp, \mfq \\ N(\mfp) N(\mfq) \leq x^{1/h}}} \frac{N(\mfp)^{-1/h} - N(\mfp)^{-1}}{N(\mfp)(1-N(\mfp)^{-1/h}+N(\mfp)^{-1})} \frac{N(\mfq)^{-1/h} - N(\mfq)^{-1}}{N(\mfq)(1-N(\mfq)^{-1/h}+N(\mfq)^{-1})}.
\end{align*}
The first sum on the right-hand side above is estimated using \lmaref{saidakeq}. For the second sum, if $X = \mathbb{Q}$, we use Parts 4 and 5 of \lmaref{boundnm} and \lmaref{sumplogp} to obtain
\begin{align*}
    & \sum_{\substack{\mfp, \mfq \\ N(\mfp) N(\mfq) \leq x^{1/h}}} \frac{N(\mfq)^{-1/h} - N(\mfq)^{-1}}{N(\mfp) N(\mfq)(1-N(\mfq)^{-1/h}+N(\mfq)^{-1})} \\
    & = \sum_{\substack{\mfp \\ N(\mfp) \leq x^{1/h}/2}} \frac{1}{N(\mfp)} \sum_{\substack{\mfq \\ N(\mfq) \leq x^{1/h}/N(\mfp)}} \frac{N(\mfq)^{-1/h} - N(\mfq)^{-1}}{N(\mfq)(1-N(\mfq)^{-1/h}+N(\mfq)^{-1})} \\
    & = \left( \mathfrak{L}_h(h+1) - \mathfrak{L}_h(2h) \right) \left(  \log \log x + \mathfrak{A} - \log h \right) + O_h \left( \frac{\log \log x}{\log x} \right),
\end{align*}
and if $X = \{ q^z \ | \ z \in \mathbb{Z} \}$, we replace $x/2$ with $x/q$ and use the corresponding result from \lmaref{sumplogp} to get the same estimate as above. Similarly, for the third sum, we obtain 
\begin{align*}
    & \sum_{\substack{\mfp, \mfq \\ N(\mfp) N(\mfq) \leq x^{1/h}}} \frac{N(\mfp)^{-1/h} - N(\mfp)^{-1}}{N(\mfp)(1-N(\mfp)^{-1/h}+N(\mfp)^{-1})} \frac{N(\mfq)^{-1/h} - N(\mfq)^{-1}}{N(\mfq)(1-N(\mfq)^{-1/h}+N(\mfq)^{-1})} \\
    & = \left( \mathfrak{L}_h(h+1) - \mathfrak{L}_h(2h) \right)^2 + O_h \left( \frac{\log \log x}{x^{1/h^2} \log x} \right).
\end{align*}
Combining the last three results with \lmaref{saidakeq}, we obtain
\begin{align*}
     & \sum_{\substack{\mfp, \mfq \\ N(\mfp) N(\mfq) \leq x^{1/h}}} \frac{1}{N(\mfp) (1-N(\mfp)^{-1/h}+N(\mfp)^{-1})} \frac{1}{N(\mfq)(1-N(\mfq)^{-1/h}+ N(\mfq)^{-1})} \notag \\
     & = (\log \log x)^2 + 2 \mathfrak{D}_1 \log \log x + \mathfrak{D}_1^2 + \mathfrak{B}
     + O_h \left( \frac{\log \log x}{\log x} \right).
\end{align*}
Combining the above with \eqref{hfullp1}, \eqref{hfullp2}, \eqref{hfullp3}, \eqref{hfullp4}, and \eqref{hfullp5}, and using the first moment completes the proof for the second moment.
\end{proof}
\section{Applications of the general setting}\label{applications}

In this section, we provide various applications of our general setting. In each case, we show that Condition ($\star$) holds, and thus deduce the distribution of the $\omega$-function over $h$-free and $h$-full elements. 

Recall the definitions of $\gamma_h$, $\mfc_1$, $\mfc_2$, $\mathfrak{D_1}$, and $\mathfrak{D_2}$ from \eqref{gammahk}, \eqref{C1}, \eqref{C2}, \eqref{D1}, and \eqref{D2} respectively. In each of the following subsections, these constants will take values depending on the prime elements of the corresponding subsection. Moreover, the $\zeta$-function in each subsection is denoted as
\[
\zeta_\mcm(s) := \sum_{\substack{\mathfrak{m}} \in \mcm \backslash \{0\}} \frac1{(N(\mathfrak{m}) )^s} 
= \prod_{\mfp \in \mcp} \Big( 1- N (\mfp)^{-s}\Big)^{-1}
\ \text{for } \Re (s) >1,
\]
where $\mcm$ and $\mcp$ respectively denote the abelian monoid and the set of prime elements defined in the subsection.


\subsection{The case of ideals in number fields}

Let $K/\mathbb{Q}$ be a number field of degree $n_K = [K: \mathbb{Q}]$ and $\mathcal{O}_K$ be its ring of integers. Let $\mathcal{P}$ be the set of prime ideals of $\mathcal{O}_K$ and $\mathcal{M}$ be the set of ideals of $\mathcal{O}_K$. Let the norm map be $N: \mathcal{M} \rightarrow \mathbb{N}$ be the standard norm map, i.e., $\mathfrak{m} \mapsto N(\mathfrak{m}) := |\mathcal{O}_K/ \mathfrak{m}|$. Let $X = \mathbb{Q}$.

 Let $\kappa_K$ be given by
\begin{equation*}
    \kappa_K = \frac{2^{r_1}(2 \pi)^{r_2} h R}{\nu \sqrt{|d_K|}},
\end{equation*}
with
\begin{align*}
    r_1 & = \text{the number of real embeddings of } K,\\
    2 r_2 & = \text{the number of complex embeddings of } K,\\
    h & = \text{the class number},\\
    R & = \text{the regulator}, \\
    \nu & = \text{the number of roots of unity}, \\
    d_K & = \text{the discriminant of } K.
\end{align*}
Landau in \cite[Satz 210]{Landau} proved that
\begin{equation*}
    \sum_{\substack{\mfm \in \mcm \\ N(\mfm) \leq x}} 1 = \kappa_K x + O \left( x^{1 - \frac{2}{n_{\scaleto{K}{3pt}}+1}}\right),
\end{equation*}
which satisfies Condition ($\star$) with $\kappa = \kappa_K$ and $\theta = 1 - \frac{2}{n_{\scaleto{K}{3pt}}+1}$. Thus, \thmref{hfreeomega} and \thmref{hfullomega} give the distribution of $\omega(\mfm)$ over $h$-free and $h$-full ideals respectively as the following:

\begin{cor}
Let $x > 2$ be a rational number. Let $h \geq 2$ be an integer. Let $\mathcal{S}_h(x)$ be the set of $h$-free ideals with norm $N(\cdot)$ less than or equal to $x$. Then, we have
$$\sum_{\mathfrak{m} \in \mathcal{S}_h(x)} \omega(\mfm) = \frac{\kappa_K}{\zeta_K(h)} x \log \log x + 
\frac{\kappa_K \mfc_1}{\zeta_K(h)} x
+ O_h \left( \frac{x}{\log x}\right),$$
and
\begin{align*}
\sum_{\mathfrak{m} \in \mathcal{S}_h(x)} \omega^2(\mfm)
& = \frac{\kappa_K}{\zeta_K(h)}  x (\log \log x)^2 + \frac{\kappa_K (2 \mfc_1 + 1)}{\zeta_K(h)} x \log \log x + \frac{\kappa_K \mfc_2}{\zeta_K(h)} x \\
& + O_h \left( \frac{x \log \log x}{\log x}\right),
\end{align*}
where $\zeta_K(s)$ is the Dedekind $\zeta$-function corresponding to the number field $K$.
\end{cor} 
\begin{cor}
Let $x > 2$ be a rational number. Let $h \geq 2$ be an integer. Let $\mathcal{N}_h(x)$ be the set of $h$-full ideals with norm $N(\cdot)$ less than or equal to $x$. Then, we have
\begin{align*}
\sum_{\mfm \in \mathcal{N}_h(x)} \omega(\mfm) & = \kappa_K \gamma_{h} x^{1/h} \log \log x + \kappa_K \gamma_h \mathfrak{D}_1 x^{1/h} + O_h \left( \frac{x^{1/h}}{\log x} \right),
\end{align*}
and
\begin{align*}
     \sum_{\mfm \in \mathcal{N}_h(x)} \omega^2(\mfm)
     & = \kappa_K \gamma_{h} x^{1/h} (\log \log x)^2 + \kappa_K \gamma_h \left( 2 \mathfrak{D}_1 + 1 \right) x^{1/h} \log \log x + \kappa_K \gamma_h \mathfrak{D}_2 x^{1/h} \\
     & \hspace{.5cm} + O_h \left( \frac{x^{1/h} \log \log x}{\log x} \right).
\end{align*}
\end{cor}
\begin{rmk}
    For the case $K = \mathbb{Q}$, we have $n_K=1$ and $\kappa_K = 1$. Thus, we can obtain the distribution of $\omega(n)$ over $h$-free and $h$-full numbers as proved in \cite[Theorems 1.1 \& 1.2]{dkl2}.
\end{rmk}

\subsection{The case of effective divisors in global function fields}

Let $q$ be a prime power and $\mathbb{F}_q$ be the finite field with $q$ elements. Let $K/\mathbb{F}_q$ be a global function field. Let $G_K$ be its genus and $C_K$ be its class number. A prime $\mfp$ in $K$ is a discrete valuation ring $R$ with maximal ideal $P$ such that $P \subset R$ and the quotient field of $R$ is $K$. The degree of $\mfp$, denoted as $\deg \mfp$, is defined as the dimension of $R/P$ over $\mathbb{F}_q$, which is finite. Let $\mcp$ be the set of all primes in $K$. Let $\mcm$ be the free abelian monoid generated by $\mcp$. More precisely, for each $\mfm \in \mathcal{M}$, we write
$$\mfm = \sum_{\mfp \in  \mathcal{P}} n_\mfp(\mfm) \mfp,$$
with $n_\mfp(\mfm) \in \mathbb{N} \cup \{ 0 \}$ and $n_\mfp(\mfm) =0$ for all but finitely many $\mfp$. We call elements of $\mcm$ as effective divisors. For an element $\mfm \in \mcm$, we define the degree of $\mfm$ as
$$\deg \mfm = \sum_{\mfp \in  \mathcal{P}} n_\mfp(\mfm) \deg \mfp.$$ 
By \cite[Lemma 5.5]{mr}, for any integer $n \geq 0$, there are finitely many effective divisors of degree $n$. This proves that $\mcp$ is a countable set that satisfies the hypothesis of our main theorems. Let the norm map $N: \mathcal{M} \rightarrow \mathbb{N}$ be the $q$-power map defined as $\mathfrak{m} \mapsto N(\mfm) := q^{\deg \mathfrak{m}}$. Let $X = \{ q^z : z \in \mathbb{Z} \}$. 

By \cite[Lemma 5.8 \& Corollary 4 to Theorem 5.4]{mr}, for a non-negative integer $n$ satisfying $n > 2 G_K - 2$, the number of effective divisors of degree $n$ is 
$$C_K \frac{q^{n- G_K+1} -1}{q-1}.$$ 
Thus, for sufficiently large $n$, we obtain
$$\sum_{N(\mfm) \leq q^n} 1 = \frac{C_K}{q^{G_K}} \left( \frac{q}{q-1} \right)^2 q^n + O(n).$$
This satisfies Condition ($\star$) with $\kappa = \frac{C_K}{q^{G_K}} \left( \frac{q}{q-1} \right)^2$ and $\theta = \epsilon$ for any $\epsilon \in (0,1)$. Thus, \thmref{hfreeomega} and \thmref{hfullomega} give the distribution of $\omega(\mfm)$ over $h$-free and $h$-full effective divisors in a global function field respectively as the following:
\begin{cor}
Let $n \in \mathbb{N}$. Let $h \geq 2$ be an integer. Let $K/\mathbb{F}_q$ be a global function field with genus $G_K$ and class number $C_K$. Let $\mathcal{S}_h(n)$ be the set of $h$-free effective divisors in $K$ of degree less than or equal to $n$. Then, we have
\begin{align*}
\sum_{\mathfrak{m} \in \mathcal{S}_h(n)} \omega(\mfm) & = \frac{C_K}{\zeta_K(h) q^{G_K}} \left( \frac{q}{q-1} \right)^2 q^n \log n + 
\frac{C_K}{\zeta_K(h) q^{G_K}} \left( \frac{q}{q-1} \right)^2 \left( \mfc_1  + \log \log q \right) q^n
\\
& \hspace{0.5cm} + O_h \left( \frac{q^n}{n}\right),
\end{align*}
and
\begin{align*}
& \sum_{\mathfrak{m} \in \mathcal{S}_h(n)} \omega^2(\mfm) \\
& = \frac{C_K}{\zeta_K(h) q^{G_K}} \left( \frac{q}{q-1} \right)^2  q^n (\log n)^2 + \frac{C_K  (2 \mfc_1 + 1 + 2 \log \log q )}{\zeta_K(h) q^{G_K}} \left( \frac{q}{q-1} \right)^2 q^n \log n \\
& \hspace{.5cm} + \frac{C_K \left( \mfc_2 + (\log \log q)^2 + (2 \mfc_1 + 1) \log \log q \right)}{\zeta_K(h) q^{G_K}} \left( \frac{q}{q-1} \right)^2 q^n + O_h \left( \frac{q^n \log n}{n}\right),
\end{align*}
where $\zeta_K(s)$ represents the $\zeta$-function corresponding to the function field $K$.
\end{cor} 
\begin{cor}
Let $n \in \mathbb{N}$. Let $h \geq 2$ be an integer. Let $K/\mathbb{F}_q$ be a global function field with genus $G_K$ and class number $C_K$. Let $\mathcal{N}_h(n)$ be the set of $h$-full effective divisors in $K$ of degree less than or equal to $n$. Then, we have
\begin{align*}
\sum_{\mfm \in \mathcal{N}_h(n)} \omega(\mfm) & = \frac{C_K \gamma_{h}}{q^{G_K}} \left( \frac{q}{q-1} \right)^2  q^{n/h} \log n + \frac{C_K \gamma_{h} (\mathfrak{D}_1 + \log \log q^{1/h})}{q^{G_K}} \left( \frac{q}{q-1} \right)^2 q^{n/h} \\
& \hspace{.5cm} + O_h \left( \frac{q^{n/h}}{n} \right),
\end{align*}
and
\begin{align*}
     \sum_{\mfm \in \mathcal{N}_h(n)} \omega^2(\mfm)
     & = \frac{C_K \gamma_{h}}{q^{G_K}} \left( \frac{q}{q-1} \right)^2 q^{n/h} (\log n)^2 \\
     & \hspace{.5cm} + \frac{C_K \gamma_{h}  \left( 2 \mathfrak{D}_1 + 1 + 2 \log \log q^{1/h} \right)}{q^{G_K}} \left( \frac{q}{q-1} \right)^2 q^{n/h} \log n \\
     & \hspace{.5cm} + \frac{C_K \gamma_{h} \left( \mathfrak{D}_2 + (\log \log q^{1/h})^2 + \left( 2 \mathfrak{D}_1 + 1 \right) \log \log q^{1/h} \right)}{q^{G_K}} \left( \frac{q}{q-1} \right)^2  q^{n/h} \\
     & \hspace{.5cm} + O_h \left( \frac{q^{n/h} \log n}{n} \right).
\end{align*}
\end{cor}
\begin{rmk}
For the special case when $K = \mathbb{F}_q(x)$, whose genus and class number are 0 and 1 respectively, we can consider the abelian monoid $A = \mathbb{F}_q[x]$, the ring of monic polynomials in one variable over $\mathbb{F}_q$. The prime elements of $A$ are the monic irreducible polynomials in $A$. The localizations of $A$ at these prime elements exhaust the set of all primes of $K$ except one, the prime at infinity. Using the fact that there are $q^n$ monic polynomials of degree $n$, we obtain
$$\sum_{\substack{N(\mfm) \leq q^n \\ \mfm \in A}} 1 = \frac{q}{q-1} q^n + O(1).$$
This satisfies Condition ($\star$) with $\kappa = q/(q-1)$ and $\theta = 0$. Thus, the distribution of $\omega(\mfm)$ over $h$-free and $h$-full polynomials over finite fields can be deduced from \thmref{hfreeomega} and \thmref{hfullomega}. Such a result will be equivalent to the ones studied by Lal\'in and Zhang \cite[Theorems 4.1 \& 6.1]{lz}. Note that, in this special case, $\kappa = q/(q-1)$ instead of $(q/(q-1))^2$ to account for the lack of the prime at infinity analog in its construction.

Moreover, for an elliptic function field $K/\mathbb{F}_q$, the genus of $K$ is 1. Thus $\kappa = C_K \left( \frac{q}{(q-1)^2} \right)$ in this case. 
\end{rmk}
\begin{rmk}
    The study of global function fields over $\mathbb{F}_q$ is geometrically equivalent to the study of irreducible projective varieties of dimension 1 over $\mathbb{F}_q$. Such varieties are also called irreducible curves. We can apply our main theorems to irreducible projective varieties of dimension r over $\mathbb{F}_q$, where $r$ is any positive integer. We study this in the following subsection.
\end{rmk}

\subsection{The case of effective \texorpdfstring{$0$}{}-cycles in geometrically irreducible projective varieties of dimension \texorpdfstring{$r$}{}}

In this subsection, we adopt notation from \cite[Example 4 of Section 4]{liuturan}. 

Let $q$ be a prime power and $\mathbb{F}_q$ be the finite field with $q$ elements. Let $r$ be a positive integer. Let $V/\mathbb{F}_q$ be a geometrically irreducible projective variety of dimension $r$. Let $\mcp$ be the set of closed points of $V/\mathbb{F}_q$, which is in bijection with the set of orbits of $V/\mathbb{F}_q$ under the action of $\textnormal{Gal}(\bar{\mathbb{F}}_q/\mathbb{F}_q)$ (see \cite[Proposition 6.9]{lorenzini}). For each $\mfp \in \mcp$, we define the degree of $\mfp$, $\deg \mfp$, to be the length of the corresponding orbit. Let $\mcm$ be the free abelian monoid generated by $\mcp$. We call elements of $\mcm$ as effective 0-cycles. For $\mfm \in \mcm$, we have $\mfm = \sum_{\mfp \in  \mathcal{P}} n_\mfp(\mfm) \mfp$ with $n_\mfp(\mfm) \in \mathbb{N} \cup \{ 0 \}$ and $n_\mfp(\mfm) =0$ for all but finitely many $\mfp$. We define the degree of $\mfm$ as
$$\deg \mfm = \sum_{\mfp \in  \mathcal{P}} n_\mfp(\mfm) \deg \mfp.$$ 
By \cite[Lemma 3.11]{lorenzini}, we deduce that $\mcp$ is countable and satisfies the hypothesis of our main theorems. Let the norm map $N: \mathcal{M} \rightarrow \mathbb{N}$ be the $q^r$-power map defined as $\mathfrak{m} \mapsto N(\mfm) := q^{r \deg \mathfrak{m}}$. Let $X = \{ q^{rz} : z \in \mathbb{Z} \}$. In \cite[Remark 1 of Section 4]{liuturan}, the third author proved that
$$\sum_{N(\mfm) \leq q^{rn}} 1 = \kappa' \left( \frac{q^r}{q^r-1} \right) q^{rn} + O_r \left( n \cdot q^{(r-1)n} \right),$$
where $\kappa'$ is some positive constant defined explicitly in \cite[Lemma 7 of Section 4]{liuturan}. This satisfies Condition ($\star$) with $\kappa = \kappa' \left( \frac{q^r}{q^r-1} \right)$ and $\theta = \epsilon$ for any $\epsilon \in (1-1/r,1)$. Thus, \thmref{hfreeomega} and \thmref{hfullomega} gives the distribution of $\omega(\mfm)$ over $h$-free and $h$-full effective 0-cycles in a geometrically irreducible projective variety of dimension $r$ as the following:

\begin{cor}
Let $n,r \in \mathbb{N}$. Let $h \geq 2$ be an integer. Let $V/\mathbb{F}_q$ be a geometrically irreducible projective variety of dimension $r$. Let $\mathcal{S}_h(n)$ be the set of $h$-free effective $0$-cycles in $V$ of degree less than or equal to $n$. Then, we have
$$\sum_{\mathfrak{m} \in \mathcal{S}_h(n)} \omega(\mfm) = \frac{\kappa'}{\zeta_V(h)} \left(\frac{q^r}{q^r-1} \right) q^{rn} \log n + 
\frac{\kappa' (\mfc_1 + \log \log q^r)}{\zeta_V(h)} \left(\frac{q^r}{q^r-1} \right) q^{rn}
+ O_h \left( \frac{q^{rn}}{n}\right),$$
and
\begin{align*}
& \sum_{\mathfrak{m} \in \mathcal{S}_h(n)} \omega^2(\mfm) \\
& = \frac{\kappa'}{\zeta_V(h)} \left(\frac{q^r}{q^r-1} \right) q^{rn} (\log n)^2 \frac{\kappa'  (2 \mfc_1 + 1 + 2 \log \log q^r )}{\zeta_V(h)} \left(\frac{q^r}{q^r-1} \right) q^{rn} \log n \\
& + \frac{\kappa' \left( \mfc_2 + (\log \log q^r)^2 + (2 \mfc_1 + 1) \log \log q^r \right) }{\zeta_V(h)} \left(\frac{q^r}{q^r-1} \right) q^{rn} + O_{h} \left( \frac{q^{rn} \log n }{n}\right),
\end{align*}
where $\zeta_V(s)$ represents the $\zeta$-function corresponding to the irreducible variety $V$.
\end{cor} 
\begin{cor}
Let $n \in \mathbb{N}$. Let $h \geq 2$ be an integer. Let $V/\mathbb{F}_q$ be a geometrically irreducible projective variety of dimension $r$. Let $\mathcal{N}_h(n)$ be the set of $h$-full effective $0$-cycles in $V$ of degree less than or equal to $n$. Then, we have
\begin{align*}
\sum_{\mfm \in \mathcal{N}_h(n)} \omega(\mfm) & = \kappa' \gamma_h \left(\frac{q^r}{q^r-1} \right) q^{rn/h} \log n \\
& \hspace{.5cm} + \kappa' \gamma_h \left(\frac{q^r}{q^r-1} \right) (\mathfrak{D}_1 + \log \log q^{r/h})   q^{rn/h} + O_h \left( \frac{q^{rn/h}}{n} \right),
\end{align*}
and
\begin{align*}
     \sum_{\mfm \in \mathcal{N}_h(n)} \omega^2(\mfm)
     & = \kappa' \gamma_h \left(\frac{q^r}{q^r-1} \right) q^{rn/h} (\log n)^2 \\
     & + \kappa' \gamma_h \left(\frac{q^r}{q^r-1} \right) \left( 2 \mathfrak{D}_1 + 1 + 2 \log \log q^{r/h} \right) q^{rn/h} \log n \\
     & + \kappa' \gamma_h \left(\frac{q^r}{q^r-1} \right) \left( \mathfrak{D}_2 + (\log \log q^{r/h})^2 + \left( 2 \mathfrak{D}_1 + 1 \right) \log \log q^{r/h} \right) q^{rn/h} \\
     & + O_h \left( \frac{q^{rn/h} \log n}{n} \right).
\end{align*}
\end{cor}

In this work, we proved the distribution of $\omega(\mathfrak{m})$ over $h$-free and $h$-full elements of a countably generated free abelian monoid. As a result, we showed that $\omega(\mathfrak{m})$ has normal order $\log \log N(\mfm)$ over these subsets. In addition, we can also establish that $\omega(\mfm)$ satisfies the Gaussian distribution over the subsets of $h$-free and $h$-full elements. We will report this result in a follow-up article. 

Let $k \geq 1$ be a natural number. Let $\omega_k(n)$ denote the number of distinct prime factors of a natural number $n$ with multiplicity $k$. The authors in \cite{dkl4} studied the distribution of $\omega_k(n)$ over $h$-free and $h$-full natural numbers. We can generalize this result to any countably generated free abelian monoid, using the setup presented in this article. The authors have been working on this and will report their findings in a future article. Note that a function field analog of this research has been studied by G\'omez and Lal\'in \cite{lg}.

The authors would like to thank the referee for their valuable comments.

\bibliographystyle{plain} 
\bibliography{mybib.bib} 

\end{document}